\documentclass[reqno]{amsart}

\usepackage{accents}
\usepackage[export]{adjustbox}
\usepackage{amsmath,amssymb,amscd,amsthm,MnSymbol,stmaryrd}
\usepackage[english]{babel}
\usepackage{chngcntr}
\usepackage{comment}
\usepackage{graphicx}
\usepackage{hyperref}
\usepackage{indentfirst}
\usepackage[utf8]{inputenc}
\usepackage{mathtools}
\usepackage{mathrsfs}
\usepackage[square,sort,comma,numbers]{natbib}
\usepackage{stmaryrd}
\usepackage{subcaption}
\usepackage{tikz-cd}
\usepackage{url}
\usepackage{xcolor}
\usepackage{todonotes}

\newcommand{\si}{\sigma}
\newcommand{\Si}{\Sigma}

\newcommand{\bC}{\mathbb{C}}

\newcommand{\bF}{\mathbb{F}}
\newcommand{\bK}{\mathbb{K}}
\newcommand{\bL}{\mathbb{L}}
\newcommand{\bP}{\mathbb{P}}
\newcommand{\bQ}{\mathbb{Q}}
\newcommand{\bR}{\mathbb{R}}
\newcommand{\bT}{\mathbb{T}}
\newcommand{\bZ}{\mathbb{Z}}

\newcommand{\blambda}{{\boldsymbol{\lambda}}}
\newcommand{\bw}{{\boldsymbol{w}}}

\newcommand{\cA}{\mathcal{A}}

\newcommand{\cC}{\mathcal{C}}
\newcommand{\cD}{\mathcal{D}}

\newcommand{\cF}{\mathcal{F}}

\newcommand{\cI}{\mathcal{I}}
\newcommand{\cL}{\mathcal{L}}
\newcommand{\cM}{\mathcal{M}}
\newcommand{\cO}{\mathcal{O}}

\newcommand{\cR}{\mathcal{R}}
\newcommand{\tcR}{{\widetilde{{\cR}}}}

\newcommand{\cT}{\mathbb{T}}
\newcommand{\cU}{\mathcal{U}}
\newcommand{\cV}{\mathcal{V}}
\newcommand{\cX}{\mathcal{X}}
\newcommand{\cIX}{\mathcal{IX}}
\newcommand{\cY}{\mathcal{Y}}

\newcommand{\cMcc}{{\mathcal{M}^{\circ}}}

\newcommand{\fp}{\mathfrak{p}}

\newcommand{\fg}{\mathfrak{g}}

\newcommand{\ch}{\mathrm{ch}}

\newcommand{\Hom}{\mathrm{Hom}}
\newcommand{\Coh}{\mathrm{Coh}}

\newcommand{\age}{\mathrm{age}}
\newcommand{\eff}{{\mathrm{eff}}}

\newcommand{\SYZ}{{\mathrm{SYZ}}}
\newcommand{\orb}{\mathrm{CR}}
\newcommand{\orbifold}{\mathrm{orb}}

\newcommand{\pr}{\mathrm{pr}}
\newcommand{\inv}{\mathrm{inv}}

\newcommand{\Pic}{ {\mathrm{Pic}} }

\newcommand{\Ker}{\mathrm{Ker}}

\newcommand{\Sh}{\mathrm{Sh}}

\newcommand{\CC}{\mathrm{CC}}

\newcommand{\Nef}{ {\mathrm{Nef}} }
\newcommand{\NE}{ {\mathrm{NE}} }
\newcommand{\BoxS}{\mathrm{Box}(\mathbf \Si)}
\newcommand{\Boxs}{\mathrm{Box}(\si)}
\newcommand{\InvM}{\mathfrak{F}^{-1}}
\newcommand{\Kah}{{\text{K\"ahler}}}
\newcommand{\id}{{\mathrm{id}}}

\newcommand{\one}{\mathbf{1}}

\newcommand{\fP}{\mathfrak{P}}

\newcommand{\sw}{\mathsf{w}}
\newcommand{\uf}{\mathsf{f}}

\newcommand{\tC}{ {\widetilde{C}} }
\newcommand{\tbT}{ {\widetilde{\bT}} }
\newcommand{\tcD}{ {\widetilde{\mathcal{D}}}}
\newcommand{\tW}{{\widetilde{W}}}
\newcommand{\tb}{{\widetilde{b}}}

\newcommand{\tM}{{\widetilde{M}}}
\newcommand{\tN}{{\widetilde{N}}}
\newcommand{\tS}{\widetilde{S}}

\newcommand{\tV}{{\widetilde{V}}}

\newcommand{\tH}{\widetilde{H}}

\newcommand{\tcY}{\widetilde{\cY}}

\newcommand{\tGamma}{{\widetilde{\Gamma}}}

\newcommand{\tch}{ {\widetilde{\mathrm{ch}}}}

\newcommand{\tNef}{\widetilde{\Nef}}
\newcommand{\tNE}{\widetilde{\NE}}
\newcommand{\tcM}{\widetilde{\cM}}
\newcommand{\tcL}{{\widetilde{\cL}}}

\newcommand{\tcMcc}{{\tcM^{\circ}}}

\newcommand{\hGamma}{\hat{\Gamma}}

\newcommand{\TODO}{\textbf{TODO:}}

\def\res{\mathop{\mathrm{Res}}}

\newcommand{\lra}{\longrightarrow}
\newcommand{\pt}{\mathrm{point}}
\newcommand{\uw}{\mathsf{w}}

\newcommand{\tbCstar}{{\widetilde{\bC^*}}}

\renewcommand{\amalg}{\sqcup}

\renewcommand{\Im}{\mathrm{Im}}
\renewcommand{\Re}{\mathrm{Re}}

\newtheorem{dummy}{dummy}[section]

\theoremstyle{definition}
\newtheorem{definition}[dummy]{Definition}
\newtheorem{remark}[dummy]{Remark}
\newtheorem*{remark*}{Remark}
\newtheorem{example}[dummy]{Example}
\newtheorem{assumption}[dummy]{Assumption}
\newtheorem{theorem}{Theorem}[section]
\newtheorem*{theorem*}{Theorem}
\newtheorem*{Theorem*}{Theorem}
\newtheorem*{example*}{Example}

\newtheorem{lemma}[theorem]{Lemma}

\newtheorem{proposition}[theorem]{Proposition}

\theoremstyle{definition}

\numberwithin{equation}{section}
\counterwithin{figure}{section}

\DeclareMathOperator{\Arg}{Arg}
\DeclareMathOperator{\C}{\mathbb{C}}

\DeclareMathOperator{\conv}{Conv}

\DeclareMathOperator{\R}{\mathbb{R}}
\DeclareMathOperator{\realpart}{Re}
\DeclareMathOperator{\spec}{Spec}
\DeclareMathOperator{\Span}{span}

\numberwithin{equation}{section}
\counterwithin{figure}{section}

\title[Mirror symmetric Gamma via Fourier]{Mirror symmetric Gamma conjecture for toric GIT quotients via Fourier transform}

\begin{document}

\begin{abstract}
    Let $\cX=[(\mathbb C^r\setminus Z)/G]$ be a toric Fano orbifold. We compute the Fourier transform of the $G$-equivariant quantum cohomology central charge of any $G$-equivariant line bundle on $\bC^r$ with respect to certain choice of parameters. This gives the quantum cohomology central charge of the corresponding line bundle on $\cX$, while in the oscillatory integral expression it becomes the oscillatory integral in the mirror Landau-Ginzburg mirror of $\cX$. Moving these parameters to real numbers simultaneously deforms the integration cycle to the mirror Lagrangian cycle of that line bundle. This computation produces a new proof the mirror symmetric Gamma conjecture for $\cX$.
\end{abstract}

\author{Konstantin Aleshkin}
\address{Konstantin Aleshkin, Kavli Institute for the Physics and Mathematics of the Universe (WPI), The University of Tokyo Institutes for Advanced Study, The University of Tokyo, Kashiwa, Chiba 277-8583, Japan}
\email{konstantin.aleshkin@ipmu.jp}

\author{Bohan Fang}
\address{Bohan Fang, Beijing International Center for Mathematical
  Research, Peking University, 5 Yiheyuan Road, Beijing 100871, China}
\email{bohanfang@gmail.com}
\author{Junxiao Wang}
\address{Junxiao Wang, Beijing International Center for Mathematical Research, Peking University, 5 Yiheyuan Road, Beijing 100871, China}
\email{wangjunxiao@bicmr.pku.edu.cn}

\maketitle

\section{Introduction}

Mirror symmetry relates Gromov-Witten invariants to more classical objects such as period integrals. For an $n$-dimensional toric Fano orbifold $\cX$, its mirror is a Landau-Ginzburg model $W^\cX:(\bC^*)^{n}\to \bC$, where the superpotential $W^\cX$ is a Laurent polynomial. The following oscillatory integral gives the \emph{quantum cohomology central charge} \cite{Iritani_2009} of the a coherent sheaf $\cL$ on $\cX$
\begin{equation}
\label{eqn:main}
\int_F e^{-\frac{W}{z}} \Omega=z^n ( \tch(\cL) \tGamma(T\cX), I_\cX)_{\mathcal{H}}.
\end{equation}
Here $\Omega$ is the standard Calabi-Yau form on the algebraic torus $(\bC^*)^n$, and the right side is the pairing in the Givental symplectic space. The $I$-function $I^\cX$ is an explicit hypergeometric function (see Definition \ref{def:I-function}), which under mirror symmetry encodes all genus $0$ descendants Gromov-Witten invariants of $\cX$, i.e. equal to the J-function. The class $\tch(\cL)$ is a modified Chern character. The Gamma class $\tGamma(\cX)$ is a modified version of the square root of the Todd class of $\cX$.

This result was proved by Iritani \cite{Iritani_2009} and later in \cite{Fang2020} the cycle $F$ was explicitly shown to be the SYZ mirror dual to $\cL.$ The quantum cohomology central charge of a coherent sheaf $\cL$ is also defined by Gromov-Witten invariants
\[
( \tch(\cL) \tGamma(T\cX), I^\cX)_\mathcal{H} =\llangle\frac{\tch(\cL) \tGamma(T\cX)}{z(z+\psi)}\rrangle_{0,1}^\cX.
\]
 The symbol $\llangle \rrangle_{0,1}^{\cX}$ is the notation of genus-zero primary and descendant Gromov-Witten potential with $1$ marked point. A key feature here is that the Gamma class $\tGamma(T\cX)$ is needed to give correct integral structures mirror to the integral cycle $F$ \cite{Hosono_2006,Iritani_2009,KKP08}.

The toric orbifold $\cX$ is given as a GIT quotient $[\bC^r\sslash_\eta G]$ with respect to certain stability condition $\eta$ where $G\cong (\bC^*)^k\subset \tbT \cong(\bC^*)^r$. The affine space $\bC^r$ together with the $G$-action has a ``$G$-equivariant mirror'' given by a \emph{$G$-equivariantly perturbed superpotential} $W^{\bC^r}_G: (\widetilde{\bC^*})^r \to \bC$, explicitly given as the following
\[
W^{\bC^r}_G=X_1+\dots+X_r - w_1 \log X_1 - w_r \log X_r.
\]
Here $w_i=w_i(\lambda_1,\dots,\lambda_k)$ under the map $H^2_{\tbT}(\pt)\to H^2_{G}(\pt)$, where $w_i$ and $\lambda_i$ are bases.

In this paper, we directly compute the following oscillatory integral
\begin{equation}
\label{eqn:integral-introduction-Cr}
\int_{F} e^{-\frac{W^{\bC^r}_G}{z}} \Omega =\tch_G(\tcL) \tGamma_G(T\bC^r).
\end{equation}
Here $F=\bR^r+(2\pi\sqrt{-1})(c_1,\dots,c_r)$ is the mirror Lagrangian cycle in $(\widetilde{\bC^*})^r\cong \bC^r$ dual to $\tcL=\cO_\chi\in K_{\tbT}(\bC^r)$, i.e. the structure sheaf $\cO$ twisted by a character $\chi=(c_1,\dots,c_r)$ in $H^*_{\tbT}(\pt).$ 
This equality is just the $G$-equivariant version of Equation \eqref{eqn:main} for $\bC^r.$

We consider the inverse analytic Fourier transform (or a Mellin transform)

\[
\prod_{a=1}^k\int_{u_a \in 1+\sqrt{-1}\bR} du_ae^{-u_a t_a} \cdot 
\]
applied to both sides of \eqref{eqn:integral-introduction-Cr} where $u_a=\lambda_a/z$ and $\Im t_a$ can be picked in a certain interval
such that the transform makes sense. The left side gives an oscillatory integral in the universal cover of a subtorus $\{\tH=t\} \subset (\widetilde{\bC^*})^r$  over a cycle $F' \cap \{\Re \tH = \Re(t)\}$. The subspace $\{\tH=t\}$ is precisely the mirror Landau-Ginzburg model of the toric GIT quotient $\cX$. 
Cycle $F'$ is a small deformation of cycle $F$ that solves $\Im(\tH) = \Im(t)$.

On the other hand, the inverse Fourier transform of the right side just gives the quantum cohomology central charge of sheaf $\cL$ descending from $\tcL$ as 
is explained in~\cite{Aleshkin_Liu_23}.

One can deform the complex parameters $t_a$ to real numbers, and must simultaneously deform $F$ so that it still lives in $H_n((\tbCstar)^r, \Re W^\cX \gg 0)$. When all $t_a$ are real, the deformed $F$ is precisely the SYZ mirror skeleton in the sense of coherent-constructible correspondence \cite{F-L-T-Z_2011}. Thus, this computation gives a quick proof of the mirror symmetric Gamma conjecture Equation~\eqref{eqn:main} (Theorem~\ref{thm:main-1} and \ref{thm:main-2}), without resorting to partial differential equations (GKZ systems).

\begin{remark*}
The relationship between the equivariant quantum cohomology of $U$ and the quantum cohomology of $U\sslash G$ by Fourier transforming the equivariant parameters to K\"ahler parameters was conjectured by Teleman \cite{Teleman2014}. Iritani proves an $S^1$-version of this conjecture in the study of quantum cohomology for blow-ups \cite{iritani_quantumblowup_2023}. Essentially Proposition \ref{prop:fourier-a} is a toric version of this conjecture, implied by Givental and Lian-Liu-Yau's mirror theorems. The use of inverse Fourier transform was briefly discussed in \cite[Section 8]{iritani2023quantumdmodulestoricvarieties}.
\end{remark*}

\subsection{Outline}
We fix the notion of toric orbifolds in Section \ref{sec:toric}, and recall the definition of the Landau-Ginzburg mirror of a toric Fano orbifold in Section \ref{sec:lg-model}. In Section \ref{sec:fourier} we compute the oscillatory integrals using inverse Fourier transforms. In Section \ref{sec:local-sys} we study the behavior of the mirror integration cycle when complex parameters move to real numbers and compare it with the mirror SYZ cycle.

\subsection{Acknowledgement}
The second author would like to thank Chiu-Chu Melissa Liu, Song Yu, and Zhengyu Zong for valuable discussion. The third author wishes to thank Nai-Chung Conan Leung and Hiroshi Iritani for helpful discussion and comments. BF and JW are supported in part by National Key R\&D Program of China 2023YFA1009803 and NSFC 12125101.

\section{Toric notion}
\label{sec:toric}

In this section, we follow the definitions in \cite[Section 3.1]{Iritani_2009}, with slightly different notation. We work over $\bC$.

\subsection{Definition}
Let $N\cong \bZ^n$ be a finitely generated free abelian group, and let $N_\bR=N\otimes_\bZ\bR$. We consider complete toric \emph{orbifolds} which have trivial generic stabilizers. A toric orbifold is given by the following data:
\begin{itemize}
\item vectors $b_1,\dots, b_{r'} \in N$. We require the subgroup $\oplus_{i=1}^{r'} \bZ b_i$ is of finite index in $N$.
\item a complete simplicial fan $\Si$ in $N_\bR$ such that the set of $1$-cones is
$$
\{\rho_1,\dots,\rho_{r'}\},
$$
where $\rho_i=\bR_{\ge 0} b_i$, $i=1,\dots, r'$.
\end{itemize}
The datum $\mathbf \Si=(\Si, (b_1,\dots, b_{r'}))$ is the \emph{stacky
  fan} \cite{BCS_2005}. We choose
\emph{additional} vectors  $b_{r'+1},\dots, b_r$ such that $b_1,\dots, b_r$ generate $N$.
There is a surjective group homomorphism
\begin{eqnarray*}
\phi: & \tN :=\oplus_{i=1}^r \bZ\tb_i & \lra  N,\\
        &  \tb_i  & \mapsto b_i.
\end{eqnarray*}
Define $\bL :=\Ker(\phi) \cong \bZ^k$, where $k:=r-n$. Then
we have the following short exact sequence of finitely generated  free abelian groups:
\begin{equation}\label{eqn:NtN}
0\to \bL  \stackrel{\psi }{\lra} \tN  \stackrel{\phi}{\lra} N\to 0.
\end{equation}
Applying $ - \otimes_\bZ \bC^*$ and $\Hom(-,\bZ)$ to \eqref{eqn:NtN}, we obtain two exact
sequences of abelian groups:
\begin{align}
\label{eqn:bT}
&1 \to G \to \tbT\to \bT \to 1,\\
&\label{eqn:MtM}
0 \to M \stackrel{\phi^\vee}{\to} \tM \stackrel{\psi^\vee}{\to} \bL^\vee \to 0,
\end{align}
where
\begin{gather*}
\bT= N\otimes_\bZ \bC^*  \cong (\bC^*)^n,\  \tbT = \tN\otimes_\bZ \bC^* \cong (\bC^*)^r,\  G = \bL\otimes_\bZ \bC^* \cong (\bC^*)^k,\\
M = \Hom(N,\bZ)  = \Hom(\bT,\bC^*), \
\tM = \Hom(\tN,\bZ)= \Hom(\tbT,\bC^*),\\
\bL^\vee = \Hom(\bL,\bZ) =\Hom(G,\bC^*).
\end{gather*}

The action of $\tbT$ on itself extends to a $\tbT$-action on $\bC^r = \spec\bC[Z_1,\dots, Z_r]$.
The group $G$ acts on $\bC^r$ via the group homomorphism $G\to \tbT$ in \eqref{eqn:bT}.

Define the set of ``anti-cones''
$$
\cA=\{I\subset \{1,\dots, r\}: \text{$\sum_{i\notin I} \bR_{\ge 0} b_i$ is a cone of $\Si$}\}.
$$
Given $I\in \cA$, let $\bC^I$ be the subvariety of $\bC^r$ defined by the ideal in $\bC[Z_1,\ldots, Z_r]$ generated by $\{ Z_i \mid i\notin I\}.$
Define the toric orbifold $\cX$ as the stack quotient
$$
\cX:=[U_\cA / G],
$$
where
$$
U_\cA:=\bC^r \backslash \bigcup_{I\notin \mathcal A} \bC^I.
$$
$\cX$ contains the torus $\cT:= \tbT/G$ as a dense open subset, and
the $\tbT$-action on $U_\cA$ descends to a $\cT$-action on $\cX$. Let $\{D_i^\bT\}$ be the basis of $\tM$ dual to $\{\tb_i\}$, and $D_i=\psi^\vee(D_i^\bT)$. 

Equivalently, one may start with a rank $k$ algebraic torus $G$, with a choice of $D_1,\dots, D_r\in \bL^\vee$ where $\bL=\Hom(\bC^*,G)$ such that $\sum_{i=1}^r \bR D_i= \bL^\vee \otimes \bR$, and a stability vector $\eta \in \bL^\vee\otimes \bR$. The set of anticones is given by $\cA = \cA_{\eta}=\{I\subset [r]: \eta \in \sum_{i\in I} \bR_{>0} D_i \}$. 

The toric quotient $\cX$ is the stacky GIT quotient w.r.t. this stability vector, i.e., $\cX=[\bC^r \sslash_\eta G].$

\subsection{Line bundles and divisors on $\cX$}
\label{sec:line-bundles}

Let $\tcD_i$ be the $\tbT$-divisor in $\bC^r$ defined by $Z_i=0$. Then
$\tcD_i \cap U_\cA$ descends to a $\bT$-divisor $\cD_i$ in $\cX$. We have
$$
\tM \cong \Pic_{\tbT}(\bC^r) \cong H^2_{\tbT}(\bC^r;\bZ),
$$
where the second isomorphism is given by the $\tbT$-equivariant first Chern class $(c_1)_{\tbT}$.  Define\footnote{The minus sign in the definition of $\uw_i$ is a convention for the purpose of matching the B-model oscillatory integrals.}
\begin{align*}
\uw_i& = -(c_1)_{\tbT}(\cO_{\bC^r}(\tcD_i)) \in H^2_{\tbT}(\bC^r;\bZ) \cong H^2_{\cT}([\bC^r/G];\bZ)\cong\tM,\\
\bar D_i^{\bT}&=(c_1)_{\bT}(\cO_{\cX}(\cD_i)),\quad \bar D_i^{\tbT}=(c_1)_{\tbT}(\cO_{\cX}(\cD_i)).
\end{align*}
Then $\{-\uw_1,\ldots,-\uw_r\}$ is a $\bZ$-basis of
$H^2_{\tbT}(\bC^r;\bZ)\cong \tM$ dual to the $\bZ$-basis
$\{ \tb_1,\ldots, \tb_r\}$ of $\tN$. Notice that $
\bar{D}_i^\cT =0$ for $i=r'+1,\dots,r$. We have
\begin{gather*}
\Pic_{\cT}(\cX)\cong H^2_{\cT}(\cX;\bZ) \cong \tM/\oplus_{i=r'+1}^r \bZ \sw_i,\\
H^2_{\cT}(\cX;\bZ) = \bigoplus_{i=1}^{r'} \bZ \bar{D}^\cT_i \cong\bZ^{r'}.
\end{gather*}
We also have group isomorphisms
$$
\bL^\vee \cong \Pic_G(\bC^r) \cong H^2_G(\bC^r;\bZ),
$$
where the second isomorphism is given by the $G$-equivariant first Chern class $(c_1)_G$. Notice that a character $\chi \in\bL^\vee$ defines a line bundle on $\cX$
$$
\cL_\chi=\{(z,t)\in \cU_A\times \bC/ (z,t)\sim (g\cdot z,\chi(g)\cdot t),g\in G\},
$$
where $G$ acts on $\cU_A$ as a subgroup of $\tbT$. There is a canonical $\tbT$-action on $\cL_\chi$ which acts diagonally on $\bC^r$ and trivially on the fiber. Define
\begin{gather*}
D_i=(c_1)_G(\cO_{\bC^r}(\tcD_i)) \in H^2_G(\bC^r;\bZ)\cong \bL^\vee,\\
\bar{D}_i=c_1(\cO_{\cX}(\cD_i))\in H^2(\cX;\bZ).
\end{gather*}
We have
$$
\Pic(\cX)\cong H^2(\cX;\bZ) \cong \bL^\vee/\oplus_{i=r'+1}^r \bZ D_i.
$$
The map
$$
\bar{\psi}^\vee: \Pic_{\cT}(\cX)\cong H^2_{\cT}(\cX;\bZ) \to \Pic(\cX)\cong H^2(\cX;\bZ)
$$
is induced from $\psi^\vee:\Pic_\tbT(\bC^r;\bZ)\cong \tM \to \Pic_G(\bC^r;\bZ)\cong \bL^\vee$ and satisfies
$$
\bar \psi^\vee(\bar{D}_i^\cT)=\bar{D}_i\quad i=1,\ldots, r'.
$$

\subsection{Torus invariant subvarieties and their generic stabilizers}
Let $\Si(d)$ be the set of $d$-dimensional cones. For each $\si\in \Si(d)$,
define
$$
I_\si=\{ i\in \{1,\ldots,r\}\mid \rho_i\not\subset \si \} \in \cA,
\quad I_\si'= \{1,\ldots, r\}\setminus I_\si.
$$
Let $\tV(\si)\subset U_\cA$ be the closed subvariety defined by the ideal of $\bC[Z_1,\ldots, Z_r]$ generated by
$$
\{Z_i=0\mid \rho_i\subset \si\} = \{ Z_i=0\mid i \in I'_\si\}.
$$
Then $\cV(\si) := [\tV(\si)/G]$ is an $(n-d)$-dimensional $\cT$-invariant closed
subvariety of $\cX= [U_\cA/G]$.

Let
$$
G_\si:= \{ g\in G\mid g\cdot z = z \textup{ for all } z\in \tV(\si)\}. 
$$
Then $G_\si$ is the generic stabilizer of $\cV(\si)$. It is a finite subgroup of $G$.
If $\tau\subset \si$ then $I_\si\subset I_\tau$, so $G_\tau\subset G_\si$. There
are two special cases. If $\si\in \Si(n)$ where $n=\dim_\bC \cX$, then $\fp_\si:= \cV(\si)$ is a $\cT$ fixed point in $\cX$, and
$\fp_\si=BG_\si$. If $\si=\rho_i \in \Si(1)$ then $\cV(\si)=\cD_i$.

\subsection{The extended nef cone and the extended Mori cone} \label{sec:nef-NE}
In this paragraph, $\bF=\bQ$, $\bR$, or $\bC$.
Given a finitely generated free abelian group $\Lambda\cong \bZ^m$, define
$\Lambda_\bF= \Lambda\otimes_\bZ \bF \cong \bF^m$.
We have the following short exact sequences of vector spaces ($\otimes_\bZ
\bF$ with Equation \eqref{eqn:NtN} and \eqref{eqn:MtM}):
\begin{eqnarray*}
&& 0\to \bL_\bF\to \tN_\bF \to N_\bF \to 0,\\
&& 0\to  M_\bF\to \tM_\bF\to \bL^\vee_\bF\to 0.
\end{eqnarray*}

Given a maximal cone $\si\in \Si(n)$, we define
$$
\bK_\si^\vee := \bigoplus_{i\in I_\si }\bZ D_i.
$$
Then $\bK_\si^\vee$ is a sublattice of $\bL^\vee$ of finite index.
We define the {\em extended nef cone} $\tNef_\cX$ as below
$$
\tNef_\si = \sum_{i\in I_\si}\bR_{\geq 0} D_i,\quad \tNef_{\cX}:=\bigcap_{\si\in \Si(n)} \tNef_\si.
$$

The {\em extended $\si$-K\"{a}hler cone} $\tC_\si$ is defined to be the interior of
$\tNef_\si$; the {\em extended K\"{a}hler cone} of $\cX$, $\tC_{\cX}$, is defined
to be the interior of the extended nef cone $\tNef_{\cX}$.

Let $\bK_\si $ be the dual lattice of $\bK_\si^\vee$; it can be viewed as an additive subgroup of $\bL_\bQ$:
$$
\bK_\si =\{ \beta\in \bL_\bQ \mid \langle D, \beta\rangle \in \bZ \  \forall D\in \bK_\si^\vee \},
$$
where $\langle-, -\rangle$ is the natural pairing between
$\bL^\vee_\bQ$ and $\bL_\bQ$. We have $\bK_\si/\bL\cong G_\si$. Define
$$
\bK:= \bigcup_{\si\in \Si(n)} \bK_\si.
$$
Then $\bK$ is a subset of $\bL_\bQ$, and $\bL\subset \bK$.

We define the {\em extended Mori cone} $\tNE_\cX\subset \bL_\bR$ to be
$$
\tNE_{\cX}:= \bigcup_{\si\in \Si(n)} \tNE_\si,\quad  \tNE_\si=\{ \beta \in \bL_\bR\mid \langle D,\beta\rangle \geq 0 \ \forall D\in \tNef_\si\}.
$$
Finally, we define extended curve classes
$$
\bK_{\eff,\si}:= \bK_\si\cap \tNE_\si,\quad \bK_{\eff}:=
\bK\cap \tNE(\cX)= \bigcup_{\si\in \Si(n)} \bK_{\eff,\si}.
$$

From now on, we make the following Fano assumption. Then the equivariant mirror theorem in \cite{C-C-I-T_2015_mirror_toric_stack,C-C-K_2015} has an explicit mirror map.

\begin{assumption}[positive (Fano) condition] \label{assump:positive}
From now on, we assume that (1) the coarse moduli toric variety is Fano and (2) we may choose $b_{r'+1},\ldots, b_r$ such that $\hat \rho:=D_1+\dots + D_r$ is contained in the closure of the extended K\"ahler cone $\tC_\cX$.
\end{assumption}

\subsection{The inertia stack and the Chen-Ruan cohomology} \label{sec:CR}
Given $\si\in \Si$, define
$$
\Boxs:=\{ v\in N: {v}=\sum_{i\in I'_\si} c_i {b}_i, \quad 0\leq c_i <1\}.
$$
If $\tau\subset \sigma$ then $I'_\tau\subset I'_\si$, so $\mathrm{Box}(\tau)\subset \Boxs$.


Given a real number $x$, let $\lfloor x \rfloor$ be the greatest integer less than or equal to  $x$,
$\lceil x \rceil$ be the least integer greater than or equal to $x$,
and then $\{ x\} = x-\lfloor x \rfloor$ is the fractional part of $x$.
Define $v: \bK_\si\to N$ by
\begin{equation}
  \label{eqn:v}
v(\beta)= \sum_{i=1}^r \lceil \langle D_i,\beta\rangle\rceil b_i=\sum_{i\in I'_\si} \{ -\langle D_i,\beta\rangle \}b_i,
\end{equation}
so $v(\beta)\in \Boxs$. Indeed, $v$ induces a bijection $\bK_\si/\bL=G_\si\cong \Boxs$. 

For any $\tau\in \Si$ there exists $\si\in \Si(n)$ such that
$\tau\subset \si$. The bijection $G_\si \to \Boxs$ restricts
to a bijection $G_\tau\to \mathrm{Box}(\tau)$. Define
$$
\BoxS:=\bigcup_{\si\in \Si}\Boxs =\bigcup_{\si\in\Si(n)}\Boxs.
$$
There is a bijection $\bK/\bL\to \BoxS$. 

Given $v\in \Boxs$, where $\si\in \Si(d)$, define $c_i(v)\in [0,1)\cap \bQ$ by
$$
v= \sum_{i\in I'_\si} c_i(v) b_i.
$$
Suppose that  $k \in G_\si$  corresponds to $v\in \Boxs$ under the bijection $G_\si\cong\Boxs$, then we have
$$
\chi^{-\sw_i}(k) = \begin{cases}
1, & i\in I_\si,\\
e^{2\pi\sqrt{-1} c_i(v)},& i \in I'_\si.
\end{cases}
$$
Define
$$
\age(k)=\age(v)= \sum_{i\notin I_\si} c_i(v).
$$
Here $-\sw_i$ are basis in $\tM=\Hom(\tbT,\bC^*)$ dual to $\tb_1,\dots,\tb_r$ (see Section \ref{sec:line-bundles}). Here $-\sw_i$ are elements in the character lattice of $\tbT$, and $\chi^{-\sw_i}(k)$ are their values for some $k\in \tbT$.

Let $IU=\{(z,k)\in U_\cA\times G\mid k\cdot z = z\}$,
and let $G$ acts on $IU$ by $h\cdot(z,k)= (h\cdot z,k)$. The
inertia stack $\cI\cX$ of $\cX$ is defined to be the quotient stack
$$
\cI\cX:= [IU/G].
$$
The inertial stack $\cIX$ comes with a projection map $\pr:\cIX\to \cX$. Note that $(z=(Z_1,\ldots,Z_r), k)\in IU$ if and only if
$$
k\in \bigcup_{\si\in \Si}G_\si \textup{ and }  Z_i=0 \textup{ whenever } \chi_i(k) \neq 1.
$$
So
$$
IU=\bigcup_{v\in \BoxS} U_v,
$$
where
$$
U_v:= \{(Z_1,\ldots, Z_r)\in U_\cA: Z_i=0 \textup{ if } c_i(v) \neq 0\}.
$$
The connected components of $\cI\cX$ are
$$
\{ \cX_v:= [U_v/G] : v\in \BoxS\}.
$$

We introduce the notation for the natural inclusion map $\iota_v \; : \; \cX_v \hookrightarrow \cX$. The involution $IU\to IU$, $(z,k)\mapsto (z,k^{-1})$ induces involutions
$\inv:\cI\cX\to \cI\cX$ and $\inv:\BoxS\to \BoxS$ such that
$\inv(\cX_v)=\cX_{\inv(v)}$. Define the $\bT$-fixed point $\fp_{\si,v}=\pr^{-1}(\fp_\si)\cap \cX_v$.

In the remainder of this subsection, we consider rational cohomology, and
write $H^*(-)$ instead of $H^*(-;\bQ)$. The Chen-Ruan orbifold cohomology, as a vector space, is defined to be \cite{Zaslow_1993, Chen-Ruan_2004}
$$
H^*_\orb (\cX)=\bigoplus_{v\in \BoxS}  H^*(\cX_v)[2\age(v)].
$$
Denote $\mathbf 1_v$ to be the unit in $H^*(\cX_v)$. Then $\mathbf 1_v\in H^{2\age(v)}_\orb(\cX)$. In particular, we have 
$$
H^0_\orb(\cX) =  \bQ \mathbf 1_0.$$
We have the following isomorphisms given by first Chern classes:
\begin{gather*}
c_1: \Pic(\cX) \xrightarrow{\sim} H^2(\cX;\bZ),\\
c_1^\bT: \Pic_\bT(\cX) \xrightarrow{\sim} H^2_\bT(\cX;\bZ).
\end{gather*}

By Assumption \ref{assump:positive}, since $\cX$ is proper, the orbifold Poincar\'e pairing on $H^*_\orb(\cX)$ is defined as
\begin{equation}\label{eqn:Poincare}
(\alpha,\beta):=\int_{\cI\cX} \alpha\cup \inv^*(\beta),
\end{equation}
We also have an equivariant pairing on $H^*_{\orb,\bT}(\cX)$:
\begin{equation}\label{eqn:T-Poincare}
(\alpha,\beta)_{\bT} := \int_{\cI\cX_{\bT}} \alpha\cup \inv^*(\beta),
\end{equation}
where
$$
\int_{ \cI\cX_{\bT}}: H_{\orb,\bT}^*(\cX) \to H_{\bT}^*(\pt) = H^*(B\bT)
$$
is the equivariant pushforward to a point. Here the dot (sometimes omitted) is the cup product in $H^*(\cIX)$. The product in the Chen-Ruan cohomology \cite{Chen-Ruan_2004} (also see \cite{Zaslow_1993} for many cases) does not explicitly appear in this paper except for the following one particular case. Namely, let $\alpha \in H^*(\cX) \subset H^*_{\orb}(\cX)$ and $\beta \in H^*(\cX_v) \subset H^*_{\orb}(\cX)$. Then the Chen-Ruan product
of $\alpha$ and $\beta$ can be described as $\alpha \cdot \beta = \iota^*_v(\alpha) \cup \beta$.

Consider Givental symplectic space $\mathcal{H} := H^*_{\orb}(\cX) \otimes_{\bC} \cO(\bC^*)$, where the second factor denotes holomorphic functions in the
loop variable $z$. We denote an element of $\mathcal{H}$ by $\alpha(z)$. The Poincar\'{e} pairings extend $z$-linearly to the symplectic space.
The latter is endowed with the skew-symmetric pairing:
\[
    (\alpha(z), \beta(z))_{\mathcal{H}} := (\alpha(z), \beta(-z)),
\]
and symplectic form $\Omega(\alpha(z), \beta(z)) := \res_{z=0}(\alpha(z), \beta(z))_{\mathcal{H}}d z$.

One can also consider the $\tbT$-equivariant cohomology. The notion $(,)_\tbT$ and $\int_{\cI\cX_\tbT}$ are self-evident.

\subsection{Chern characters, Gamma classes and central charges}
In this subsection $X$ is a general smooth DM stack with a linear action of a complex reductive group $G$.

\subsubsection{Orbifold Chern characters and twisted Chern characters} Let $K_G(X)$ be the Grothendieck group of 
(topological) vector bundles on $X$. We consider orbifold equivariant Chern character
\[
    \ch_{G} \; : \; K_G(X) \to H^*_{\orb, G}(X)\hat{},
\]
where $H^*_{\orb, G}(X)\hat{}$ is the completion of $H^*_{\orb, G}(X)\hat{}$ by the maximal homogeneous ideal. By the splitting principle,
it is enough to define the Chern character for K-classes of equivariant line bundles $\cL$. Let $\cI X = \bigsqcup_{v \in \pi_0(\cI X)}  X_v$ and $\iota_v:X_v
{\hookrightarrow} X$ be the inclusion map.
\begin{equation}
    \ch_{G}(\cL) = \sum_{v \in \pi_0(\cI X)} e^{2\pi \sqrt{-1} \age_v(\cL)}e^{c_1^{G}(\iota^*_v\cL)} \one_v,
\end{equation}
where $\age_v(\cL)$ is a unique rational number in $[0,1)$ such that the stabilizer of $X_v$ acts in each fiber of $\iota^*_v\cL$ by $e^{2\pi \sqrt{-1} \age_v(\cL)}$.

In this paper we define the twisted Chern character as follows.
\begin{definition}

\[
    \widetilde\ch_G(\cL) = \sum_{v \in \pi_0(\cI X)} (\widetilde\ch_G)_v(\cL) \one_v =  \sum_{v \in \pi_0(\cI X)} e^{2\pi \sqrt{-1} \age_v(\cL)}e^{{\color{black} -\frac{2\pi\sqrt{-1}}{z}c_1^{G}(\cL)}} \one_v.
\]
\end{definition}

\subsubsection{Gamma classes and central charges}

Following Iritani, we define the orbifold Gamma class as a multiplicative characteristic class
\[
    \hat\Gamma_G \; : \; K_G(X) \to H^*_{\orb, G}(X)\hat{},
\]
such that for a $G$-equivariant line bundle $\cL$ on $X$ we have:
\[
    \hat\Gamma_G(\cL) = \bigoplus_{v \in \pi_0(\cI X)} \hat\Gamma(1 - \age_v(\cL) + c_1^{G}(\cL)).
\]

Correspondingly, we define the $z$-twisted Gamma class by
\begin{equation}
    \widetilde{\Gamma}_G(\cL) = \bigoplus_{v \in \pi_0(\cI X)} z^{-\age_v(\cL)}z^{c_1^G(\cL)/z}\Gamma(1-\age_v(\cL)+c_1^G(\cL)/z).
\end{equation}
When $G$ is trivial (in the non-equivariant setting), we drop the letter $G$ and denote these classes as $\tch$, $\hGamma$ and $\tGamma$.

Let $I(z) \; : \; \bC^* \to H^*_{\orb, G}(X)$ be an analytic function of variable $z$ with values in the orbifold cohomology.
A central charge associated to $I(z)$ of $\cF \in K_G(X)$ is defined to be
\begin{equation}
    Z_{A}^{X, G}(\cF; I) := \frac{z^{\dim X}}{(2\pi\sqrt{-1})^{\dim X}}(\widetilde\Gamma_G(T X) \tch_G(\cF), I(z))_{\mathcal{H}}.
    \label{eqn:Z_A}
\end{equation}

\begin{remark*}
    In the case where $I(z) = L^*(\tau,-z) = J(\tau, z)$ is the Givental J-function in the notation of~\cite{Iritani_2009}, the central charge defined above
    corresponds to the central charge of~\cite{Iritani_2009}.
\end{remark*}

\section{Mirror Landau-Ginzburg model}
\label{sec:lg-model}

In this section, we define a Landau-Ginzburg model for the mirror of $\cX$. For any algebraic torus $H$, we denote $M^H=\Hom(H,\bC^*)$, so $\bL^\vee=M^G$, $M=M^\bT$ and $\tM=M^\tbT$. 

Apply the exact functor $\Hom(-,\bC^*)$ to the short exact sequence
\eqref{eqn:NtN} and we get
\[
1\to \Hom(N,\bC^*)=M_{\bC^*} \to (\bC^*)^r=\tM_{\bC^*} \stackrel{\mathfrak{p}}{\longrightarrow} \cM=\Hom(\bL,\bC^*)\to 1.
\]
We see that $\cY_q:=\mathfrak{p}^{-1}(q)\cong (\bC^*)^k$ is a subtorus in $\tM_{\bC^*}$. 
We find an integral basis $p_1,\dots,p_k \in \bL^\vee \cap \tNef_\cX$ such that $p_{k'+1},\dots,p_k\in \sum_{i=r'+1}^r \bR_{\geq 0} D_i$ where $k' = k-(r-r')$. As discussed in \cite[p1037]{Iritani_2009}, this choice is always possible. Define the \emph{charge vectors}
\[
l^{(a)}=(l_1^{(a)},\dots, l_r^{(a)})\in \bZ^r,\quad
\psi(e_a)=\sum_{i=1}^r l_i^{(a)} \tb_i.
\]
with $\{e_a\}$ being the basis of $\bL$ dual to $\{p_a\}$. So
\[
D_i=\psi^\vee(D_i^\bT)=\sum_{a=1}^k l_i^{(a)}p_a,\quad i=1,\dots,r.
\]
Therefore, $\cY_q$ could be written as
$$
\cY_q=\{(X_1,\dots,X_r)\in (\bC^*)^r|H_a:=\prod_{i=1}^r X_i^{l^{(a)}_i}=q_a,\ a=1,\dots,k\},
$$
where $q=(q_1,\dots, q_k)$ are coordinates on $\cM$, as \emph{complex parameters}. We have a flat family of algebraic tori $\cY\subset (\bC^*)^r \times \cM$. For any $\beta\in \bK$, denote $q^\beta=\prod_{a=1}^k q_a^{\langle \beta, p_a \rangle}$. This object may involve fractional powers since $\beta \in \bK$ in which $\bL$ is a sublattice.

Let $\tcM=M^G_\bC \cong \bC^k$ be the universal cover of $\cM$, and let $X_i=e^{x_i}$, $q_i=e^{t_i}$. Denote the family
\[
\tcY= \cY\times_{\tM_{\bC^*}\times \cM} (\tM_\bC\times \tcM),
\]
where $\cY\hookrightarrow \tM_{\bC^*} \times \cM$, and the map $ \tM_\bC \times  \tcM \to \tM_{\bC^*} \times \cM$ is given by $x_i \mapsto e^{x_i}, t_i \mapsto e^{t_i}.$ The fiber $\tcY_t$ is given by
\[
\tH_a:=\sum_i {l_i^{(a)} x_i}=t_a.
\]
There is a covering map $p: \tcY_t \to \cY_{e^t}$.

For any sub-torus $K$ of $\tbT$, a $K$-\emph{equivariant framing} is a choice of group morphism $\uf: M^K \to \bC$. Since the embedding $K\hookrightarrow \tbT$ induces a linear map $\tM_\bC \to M^K_\bC$, $\uf$ induces a linear map $\tM_\bC\to \bC$, and we still denote it by $\uf$.

\begin{definition}
The $H^2(\cX)$-deformed superpotential on $\cY_q$ is
$$
W^\cX(X)=\sum_{i=1}^r X_i.
$$
The $K$-equivariant superpotential on $\tcY$ is
\[
\tW^\cX_{K} = W(e^x) - \uf(x).
\]
\end{definition}

\begin{example}
    A $\tbT$-equivariant framing is given by $\uf(\bw_i)=w_i$, and the $\tbT$-equivariant superpotential is
    \[
        \tW^\cX_\tbT=X_1+\dots +X_r -w_1 x_1 -\dots -w_r x_r
    \]
    defined on $\tcY$.
\end{example}

\begin{example}
    Let $\cX=[\bC^r\sslash_\eta G]$ be a toric GIT quotient stack. Then 
    \begin{equation}
     \tW^{\bC^r}_G=X_1+\dots +X_r  -\sum_{a=1}^k \lambda_a \sum_{i=1}^r l_i^{(a)}x_i
     \label{eqn:WG}
    \end{equation}
    on $\tM_\bC$, while
    \[
    W^\cX=X_1+\dots +X_r
    \]
    on $\cY$.
\end{example}

Let 
\[
\Omega_{\tM_{\bC^*}}=\frac{dX_1}{X_1}\dots \frac{dX_r}{X_r}
\]
be the torus invariant Calabi-Yau form on $\tM_{\bC^*}$. Its pullback to $\tM_{\bC}$ is $\Omega_{\tM_\bC}=dx_1\dots dx_r$. 

We define~\footnote{One has to choose an orientation convention such as $\Omega_{\tM_{\bC^*}}=\Omega_{\cY_q}\wedge d\log H_1\wedge \ldots \wedge d\log H_k$, but
the precise choice will not play a role.} 
\begin{align*}
\Omega_{\cY_q}&= \frac{\Omega_{\tM_{\bC^*}}}{d\log H_1\wedge \ldots \wedge d\log H_k},\\
\Omega_{\tcY_t}&= \frac{\Omega_{\tM_{\bC}}}{d \tH_1 \wedge \ldots \wedge d\tH_k}.
\end{align*}
We have $p^* \Omega_{\cY_q}= \Omega_{\tcY_t}$.

For certain $t\in \tcM$, a $K$-equivariant framing and a locally complete submanifold $F\subset \tcY_t$, we define the B-model equivariant quantum cohomology central charge
\[
Z_B^{\cX,K}(F)=\int_F e^{-\frac{\tW^\cX_K}{z}} \Omega_{\tcY_t}.
\]
for $F$. We allow this $F$ to be a non-compact manifold as long as the integral converges absolutely for $\Re\lambda_a>0$ and $z>0$.

Similarly the B-model non-equivariant central charge is given by 
\begin{equation} \label{eq:bCentralCharge1}
    Z^{\cX}_B(F)=\int_F e^{-\frac{W^\cX}{z}}\Omega_{\cY_q}
\end{equation}
for given $q\in \cM$ and the locally complete submanifold $F\subset \cY_q$ if the integral converges absolutely for $z>0$.

\section{B-model oscillatory integral for toric GIT quotients}
\label{sec:fourier}

For a toric Fano orbifold $\cX$ defined in Section \ref{sec:toric}, it is given as a GIT quotient $\cX=[\bC^r\sslash_\eta G]$. The $G$-equivariant superpotential is given by Equation \eqref{eqn:WG}. We use a G-equivariant framing $\uf:M^G\to \bC$ such that $\Re(\lambda_a)=\Re(\uf(\blambda_a))>0$. We also require $z>0$ in our computation.

Consider a real $\tbT$-equivariant Picard element $\tcL=\cO(\sum_{i=1}^r c_i\tcD_i) \in \Pic_\tbT(\tM_\bC)_\bR \cong \tM_\bR$. Its mirror Lagrangian cycle is
\[
F_\tcL= \prod_{i=1}^r \left( \bR+2\pi\sqrt{-1} c_i \right).
\]
The B-model central charge is (noting the equivariant $I$-function is $\one$.)
\begin{align}
Z_B^{\tM_\bC,G}(F_\tcL)&=\int_{F_\tcL} e^{-W_G^{\bC^r}/z} \Omega_{\tM_\bC}, \nonumber \\
&=\prod_{i=1}^r e^{-\frac{2\pi\sqrt{-1}c_i \sum_{a=1}^k \lambda_a l_i^{(a)}}{z}}\int_{X_i>0}e^{-\frac{X_i}{z}} X_i^{\frac{\sum_{a=1}^k \lambda_a l_i^{(a)}}{z}-1}dX_i\nonumber \\
&= \prod_{i=1}^r e^{-\frac{2\pi\sqrt{-1}c_i \sum_{a=1}^k \lambda_a l_i^{(a)}}{z}} \cdot  \prod_{i=1}^r z^{\frac{\sum_{a=1}^k\lambda_a l_i^{(a)}}{z}}\Gamma\left(\frac{\sum_{a=1}^k \lambda_a l_i^{(a)}}{z}\right)\\
&= z^{r}\prod_{i=1}^r e^{-\frac{2\pi\sqrt{-1}c_i \sum_{a=1}^k \lambda_a l_i^{(a)}}{z}} \cdot  \prod_{i=1}^r z^{\frac{\sum_{a=1}^k\lambda_a l_i^{(a)}}{z}}\frac{\Gamma\left(1+\frac{\sum_{a=1}^k \lambda_a l_i^{(a)}}{z}\right)}{\sum_{a=1}^k \lambda_a l_i^{(a)}}\\
&= z^r\tch_G(\tcL) \frac{\widetilde\Gamma_G(T\bC^r)}{e_G(T\bC^r)}\Big\vert_{\blambda_a=\lambda_a}= (2\pi \sqrt{-1})^{r}Z_A^{\bC^r,G}(\tcL)\vert_{\blambda_a=\lambda_a}.\nonumber
\end{align}


Let $\cL$ be the image of $\tcL$ under the natural map $\mathrm{Pic}_{\bT}(\cX) \to \mathrm{Pic}(\cX)$, that is the non-equivariant
limit of $\tcL$. The central charge only depends on $\cL$. Indeed, if $\tcL$ is represented by $\sum_{i=1}^{r'} c_i \bar{D}^\bT_i \in \mathrm{Pic}_{\bT}(\cX)$, then $\cL$ is represented by
\begin{equation} \label{eq:ha}
    \bar{\psi}^{\vee}(\sum_{i=1}^{r'} c_i \bar{D}^\bT_i) = \sum_{i=1}^{r'} c_i \bar{D}_i = \sum_{i=1}^{k'} h_a \bar{p}_a \in \mathrm{Pic}(\cX),
\end{equation}
where $h_a = \sum_{i=1}^{r'}l_i^{(a)}c_i$. The central charge only depends on $\tcL$ via the combinations $h_a$.

\paragraph{\bf Fourier transform}

Let $f(u)$ be a function holomorphic in a strip $[\alpha,\beta]\times\sqrt{-1}\bR$. 

Its inverse Fourier transform is an analytic function defined by the following integral
\[
    \InvM_{u} (f)(t) = (2\pi\sqrt{-1})^{-1}\int_{\gamma+\sqrt{-1}\bR} f(u)e^{ut}du, \quad \gamma \in [\alpha,\beta],
\]
when it converges absolutely (we assume that it converges absolutely and uniformly for $\gamma \in [\alpha,\beta]$ such that the integral
does not depend on the particular choice.

More generally, let $f(u)$ be a holomorphic function defined in a domain $\prod_{a=1}^k U\times\sqrt{-1}\bR$.
We define its inverse Fourier transform as an iterated Fourier transform assuming absolute uniform convergence:
\[
    \InvM_u(f)(t) = (2\pi\sqrt{-1})^{-k}\int_{\gamma + \sqrt{-1}\bR^k} f(u) e^{\sum_{a=1}^k u_a t_a} \prod_{a=1}^k du_k, \quad \gamma \in U \subset \bR^k.
\]

We can compute the inverse Fourier transform of the central charge using the Fourier inversion theorem. 
Note, that $Z^{\widetilde{M}_{\bC},G}_B(\cF_{\widetilde{L}})$ is meromorphic in the domain $u \in (0,\infty)^{k}$. Moreover, using Stirling's approximation:
\[
    \Gamma(z) = O(e^{-\frac{\pi}2|\Im(z)|}), \quad |\Im(z)| \to \infty
\]
uniformly for any compact interval in $(0,\infty)$ we see that $Z^{\widetilde{M}_{\bC},G}_B(\cF_{\widetilde{L}})$ has uniformly
exponential decay for $\Re(u)$ in any compact subset of the positive orthant, more precisely, let $\lambda_a/z= u_a$. Then
\begin{equation} \label{eq:centralAsymptotics}
    Z^{\widetilde{M}_{\bC},G}_B(\cF_{\widetilde{L}}) = e^{-2\pi\sqrt{-1}\sum_{i=1}^rc_i\sum_{a=1}^ku_a l^{(a)}_i}O\left(\exp\left(-\frac{\pi}{2}\sum_{i=1}^r\left|\sum_{a=1}^k\mathrm{Im}(u_a) l^{(a)}_i\right|\right)\right), \quad |\mathrm{Im}(u)| \to \infty.
\end{equation}

Thus, we can compute its inverse Fourier transform with
respect to $u$. We do this by noticing that we can write the central charge itself as a Fourier transform in $\widetilde{H}_a = \sum_i l^{(a)}_i x_i$.

First, we can deform the contour $F_{\tcL}$. Let $c_i' \in \bR$ be a deformation of $c_i$ such that $|2\pi c'_i-2\pi c_i|< \pi/2$ or $|c'_i -c_i| <1/4$. Then,
we can replace $F_{\tcL}$ in the definition of $Z^{\widetilde{M}_\bC,G}_B(F_{\tcL})$ with 
\begin{equation} \label{eq:fDeformation}
    F_{\tcL'} = \prod_{i=1}^r \left(\bR+2\pi\sqrt{-1}c_i'\right),  
\end{equation}
homotopic to $F_{\tcL}$.

Using absolute convergence of the integral, we can rewrite the central charge as 
\begin{align} 
    &Z_B^{\tM_\bC,G}(F_{\tcL'}) = \int_{F_{\tcL'}} e^{-\frac{X_1+\dots+X_r}{z}} \exp(\sum_{i=1}^r l_i^{(a)} x_i u_a) \Omega_{\tM_\bC} \nonumber  \\
    = &\left(\prod_{a=1}^k\int_{ (\bR +2\pi\sqrt{-1}\sum_{i=1}^r l^{(a)}_i c'_i)} dt_a e^{t_a u_a}\right) \int_{F_{\widetilde{\cL}'} \cap \bigcap_{a=1}^k \{\Re\widetilde{H}_a = \Re (t_a) \}} e^{-\frac{X_1 + \cdots + X_r}{z}} \frac{\Omega_{\widetilde{M}_{\bC}}}{d\widetilde{H}_1 \wedge \cdots \wedge d\widetilde{H}_k} \nonumber \\ 
    = &\left(\prod_{a=1}^k\int_{ (\bR +2\pi\sqrt{-1}\sum_{i=1}^r l^{(a)}_i c_i')} dt_a e^{t_a u_a}\right) Z^{\cX}_B(F_{\widetilde{\cL}'}
    \cap \bigcap_{a=1}^k\{\Re\widetilde{H}_a=\Re(t_a)\}).
\end{align}
The first equality is the Fubini theorem, and the second one is by defining formula~\eqref{eq:bCentralCharge1} for B-model central charges on $\cX$.
Note, that imaginary part of $t$ is determined by the shift $\{c'_i\}_i$ in the contour $F_{\cL'}$.

The multiple application of the analytic Fourier inversion theorem~\ref{thm:fourierInversion} implies that 
\begin{equation} \label{eq:B-charge_mellin_transform}
    \InvM_u (Z_B^{\tM_\bC,G}(F_\tcL)) = Z^{\cX}_B(F_{\widetilde{\cL}'}
    \cap \bigcap_{a=1}^k\{\Re\widetilde{H}_a=\Re(t_a)\}),
\end{equation}
where the right-hand side is an absolutely convergent integral for 
\begin{equation} \label{eq:tStrip}
    \{\sum_{i}l_i^{(a)}(2\pi c_i -\pi/2) < \Im(t_a) = 2\pi \sum_i l_i^{(a)}  c'_i < \sum_i l_i^{(a)}(2\pi c_i + \pi/2)\}
\end{equation}
as follows from the asymptotics~\eqref{eq:centralAsymptotics}.

Thus, we have shown the following.

\begin{proposition}
\label{prop:fourier-b}
\[
   \InvM_{u}(Z_B^{\tM_\bC,G}(F_\tcL))= Z^{\cX}_B(F_{\tcL'}\cap  \bigcap_{a=1}^k \{\Re\tH_a=\Re(t_a)\}).
\]
\end{proposition}

We also have an A-model counterpart of the last statement.
The non-equivariant small $I$-function of $\cX$ is~\cite{Iritani_2009}
\begin{equation}
\label{def:I-function}
    I^{\cX}(t,z) = e^{\sum_{a=1}^k t_a p_a/z}\sum_{\beta \in \bK_{\eff}} e^{\sum_{a=1}^k t_a \beta_a} \prod_{i=1}^r\frac{\prod_{\{c\} = \{D_i(\beta)\}, \; c \le 0} (\bar D_i+cz)}{\prod_{\{c\} = \{D_i(\beta)\}, \; c \le D_i(\beta)} (\bar D_i+cz)} \one_{v(\beta)} 
\end{equation}

The A-model central charge in this paper is the central charge associated
with the I-function for any line bundle $\cL\in \Pic(\cX)$ as given by Equation \eqref{eqn:Z_A}
\[
    Z_A^{\cX}(\cL) := Z_A^{\cX,G=\{1\}}(\cL;I^\cX).
\]

\begin{lemma}
    Explicitly, the A-model central charge is
    \begin{multline}
        Z_A^{\cX}(\cL) = (2\pi\sqrt{-1})^{-n}\sum_{v \in \BoxS}\sum_{\substack{\beta \in \bK_{\eff}\\D_i(\beta) = v_i}}z^{\dim \cX_v-k}\int_{\cX_v}
        e^{-\sum_a t_a(p_a/z-\beta_a)}z^{\sum_{i=1}^r\bar D_i/z-D_i(\beta)}
        \times \\ \times
        \prod_{i, v_i = 0}\bar D_i\prod_{i=1}^r\Gamma(\bar D_i/z-D_i(\beta)) \, {\tch}_v(\cL)
    \end{multline}
\end{lemma}
\begin{proof}
    It is convenient to rewrite the I-function using the shift property of gamma functions $\Gamma(x+1) = x\Gamma(x)$:
    \begin{multline}
        I^{\cX}(t,z)
        = e^{\sum_{a=1}^k t_a p_a/z}\sum_{\beta \in \bK_{\eff}} e^{\sum_{a=1}^k t_a \beta_a} z^{-\sum_i D_i(\beta)+\{-D_i(\beta)\}} \times \\ \times \prod_{i, D_i(\beta)\in \bZ}\frac{\bar D_i}{z}\prod_{i=1}^r 
        \frac{\Gamma(\bar D_i/z+\{D_i(\beta)\})}{\Gamma(1+\bar D_i/z+D_i(\beta))} \one_{v(\beta)} 
    \end{multline}
    where curly braces denote the fractional part of a real number.
    We also use the reflection relation for the gamma functions $\Gamma(x)\Gamma(1-x)\sin(\pi x) = \pi$ to restate the last formula as \begin{multline}
        I^{\cX}(t,z) = e^{\sum_{a=1}^k t_a p_a/z}\sum_{\beta \in \bK_{\eff}} e^{\sum_{a=1}^k t_a \beta_a} z^{-\sum_i D_i(\beta)+\{-D_i(\beta)\}} \times \\ \times \prod_{i, D_i(\beta)\in \bZ}\frac{\bar D_i}{z}\prod_{i=1}^r(-1)^{1+D_i(\beta)-\{D_i(\beta)\}} 
        \frac{\Gamma(-\bar D_i/z-D_i(\beta))}{\Gamma(1-\{D_i(\beta)\}-\bar D_i/z)} \one_{v(\beta)}.
    \end{multline}
    We then compute $I^{\cX}(t, -z)$:
    \begin{multline}
        I^{\cX}(t,-z) = e^{-\sum_{a=1}^k t_a p_a/z}\sum_{\beta \in \bK_{\eff}} e^{\sum_{a=1}^k t_a \beta_a} z^{-\sum_i D_i(\beta)+\{-D_i(\beta)\}} \times \\ \times \prod_{i, D_i(\beta)\in \bZ}\frac{\bar D_i}{z}\prod_{i=1}^r 
        \frac{\Gamma(\bar D_i/z-D_i(\beta))}{\Gamma(1-\{D_i(\beta)\}+\bar D_i/z)} \one_{v(\beta)}.
    \end{multline}
    We also compute
    \[
        \widetilde{\Gamma}(T\cX){\tch}(\cL) = 
        \sum_{v \in \BoxS}z^{-\sum_{i=1}^r v_i}\prod_{i=1}^r \Gamma(1-v_i+\bar D_i/z) {\tch}_v(\cL)\one_{v}.
    \]
    Integrating the product of the two contributions above we get the claim.    
\end{proof}

\begin{proposition}
    There exists a nonempty open region $U \subset (\bL^{\vee}\otimes \bC) \times \widetilde\bC^*$ such that for $(t,z) \in U$ and $\tcL\in \tM$ the inverse Fourier transform of $Z_A^{\bC^r,G}(\tcL)\vert_{\blambda_a=\lambda_a}$ is convergent and is equal to the non-equivariant central charge for $\cX=[\bC^r\sslash_\eta G]$ up to a
    normalization factor.
    \[
       \InvM_{u}(Z^{\bC^r, G}_A(\widetilde\cL)|_{\blambda_a=\lambda_a}) =  (2\pi \sqrt{-1})^{-k}Z_A^{\cX}(\cL)=  \frac{z^{n}}{(2\pi\sqrt{-1})^{r}}(\tch(\cL) \tGamma(T\cX) , I^\cX(t,z))_{\mathcal{H}}.
    \]
    where $\cL$ is an image of $\tcL$ under the map $\tM \to \bL^\vee \to H^2(\cX;\bZ)=\bL^\vee/\sum_{j=r'+1}^r \bZ D_j$.
    \label{prop:fourier-a}
\end{proposition}

    
\begin{proof}
This statement is a modification of the theorem~\cite{Aleshkin_Liu_23}[5.6] called the Higgs-Coulomb correspondence there. In order for the proof of the theorem
to work, we need to take care of convergence issues and $z$-dependence.

\begin{lemma}
    The integral
    \begin{equation}
        \InvM_{u}(Z^{\bC^r, G}_A(\widetilde\cL)|_{\blambda_a=\lambda_a}) = \prod_{a=1}^k (2\pi \sqrt{-1})^{-1}\int_{1+\sqrt{-1}\bR} du_a e^{u_a t_a} Z_A^{\bC^r,G}(\tcL)\vert_{\blambda_a=\lambda_a}
    \end{equation}
    is absolutely convergent if the following so-called grade restriction rule is satisfied
   
    \[
        |\sum_{a=1}^k (\Im(t_a)- 2\pi h_a)\nu_a| < \frac{\pi}{2}\sum_{i=1}^r |D_i(\nu)|, \;\;\; \forall \nu \in \bL\otimes\bC \backslash \{0\},
    \]
    where $h_a = \sum_{i=1}^r l_i^{(a)}c_i$ is defined in~\eqref{eq:ha}. Note, that this is the same condition as appears 
    in-\eqref{eq:tStrip} in the convergence analysis of the oscillatory integrals.
  
\end{lemma}
We note that the grade restriction rule is always satisfied if $\Im(t_a) = {\color{black}h_a}$.
\begin{proof}
    This follows lemma~\cite{Aleshkin_Liu_23}[5.8] word by word. We used the fact that $z$ is real positive and the fact that the Calabi-Yau condition of \cite{Aleshkin_Liu_23} does not play a role in this statement.
\end{proof}

Let us explain how to modify the proof of theorem~\cite{Aleshkin_Liu_23}[5.6] in our case. The main difference of this paper is that the coarse moduli 
space $X$ is assumed to be Fano, and $z$-dependence in the central charges. The $z$-dependence can be taken care of by doing the following 
coordinate change $u_a = \lambda_a/z, \; \tilde t_a = t_a - \log z\sum_{i=1}^r \lambda_a l^{(a)}_i$. After this coordinate change the equivariant
central charge becomes
\begin{align*}
    &\InvM_{\lambda} (Z_A^{\bC^r,G}(\tcL)\vert_{\blambda_a=\lambda_a})(t)\\
    = &\frac{z^r}{(2\pi \sqrt{-1})^{r} }\prod_{a=1}^k (2\pi \sqrt{-1})^{-1}\int_{1+\sqrt{-1}\bR} du_a e^{(t_a  {\color{black}- \sqrt{-1}h_a}) u_a } \prod_{i=1}^r z^{\sum_{a=1}^k  u_a l_i^{(a)}} \Gamma\left(\sum_{a=1}^k  u_a l_i^{(a)}\right)\\
    = & \frac{z^r}{(2\pi \sqrt{-1})^{r}} \prod_{a=1}^k (2\pi \sqrt{-1})^{-1}\int_{1+\sqrt{-1}\bR} du_a e^{(\tilde t_a {\color{black}- \sqrt{-1}h_a}) u_a } \prod_{i=1}^r  \Gamma\left(\sum_{a=1}^k  u_a l_i^{(a)}\right).
\end{align*}
We see that the $z$-dependence completely factors out of the computation. We also note that the integration contour choosen here satisfies the condition from the
definition~\cite{Aleshkin_Liu_23}[5.2] with $\alpha_i = 0$ by the choice of the basis.\\

The fact that $\sum_i D_i \ne 0$ requires a more careful analysis as it affects the convergence statements used in the proof.
The first affected statement is the proposition~\cite{Aleshkin_Liu_23}[A.3] that regulates the convergence of the central charge $Z^{\cX}_A(\cL')$.
The second statement that needs modification is corollary~\cite{Aleshkin_Liu_23}[5.12] that uses convergence proposition~\cite{Aleshkin_Liu_23}[A.5].
Essentially, Fano conditions makes the convergence even better than in the Calabi-Yau case. 
We use the lemma~\ref{lem:convergence} of the appendix to replace the propositions A.3 and A.5 in~\cite{Aleshkin_Liu_23}. 

Applying the non-equivariant limit of the modification of the theorem~\cite{Aleshkin_Liu_23}[5.6] and the non-equivariant limit of the theorem~\cite{Aleshkin_Liu_23}[4.20], and restoring the $z$-dependence we obtain the claimed formula:
\begin{equation}
    \InvM_{\lambda} (Z_A^{\bC^r,G}(\tcL)\vert_{\blambda_a=\lambda_a})(t) =  \frac{z^{n}}{(2\pi\sqrt{-1})^{r}}(\widetilde\Gamma(T\cX) \widetilde\ch(\cL), I(t, z))_{\mathcal{H}}
\end{equation}
in the non-empty common domain of convergence of the LHS and the RHS. The prefactor comes from comparing the normalizations of the central charges.

\end{proof}

Combining Proposition \ref{prop:fourier-a} and \ref{prop:fourier-b}, we obtain the following.
\begin{theorem}
\label{thm:main-1}

Let $\tcL\in \tM$ and $t \in \widetilde{\cM}$ satisfy the convergence condition~\eqref{eq:tStrip}.  Let $\cL$ be the image of $\tcL$ under $\tM \to \bL^\vee \to H^2(\cX;\bZ)$, and $F_{\tcL'}$ be the deformation of $F_{\tcL}$ compatible with $\Im(t)$ as in~\eqref{eq:fDeformation}. Then, we have the following mirror symmetry statement
for the central charges:
\begin{equation}
    Z_A^{\cX}(\cL)=(2\pi \sqrt{-1})^{-n}Z^{\cX}_B(F_{\tcL'}\cap  \bigcap_{a=1}^k \{\Re\tH_a=\Re(t_a)\}).
\end{equation}

\end{theorem}

\section{Local system of the mirror cycles}
\label{sec:local-sys}

We list some facts from \cite[Section 3]{Iritani_2009}:
\begin{itemize}
    \item There is an open dense ``good'' part $\cM^{\circ}$ of $\cM$ such that $$R^\vee_{\bZ,q}:=H_n(\cY_q;\Re(W^\cX)\gg 0;\bZ)$$ for $q\in \cMcc$form a local system $\cR^\vee$ of rank $N$ \cite[Proposition 3.12]{Iritani_2009}.
    \item There exists an $\epsilon>0$ such that $\{0<|q|<\epsilon\}\subset \cMcc$ \cite[Lemma 3.8]{Iritani_2009}.
\end{itemize}
We denote $\tcR^\vee \to  \tcMcc$ be the pullback. For a locally complete submanifold $F\in \tM_{\bC}$, if it represents an $[F]\in H_n(\cY_q;\Re(W^\cX)\gg 0)$, then the integral defining the central charge $Z_B^\cX(F)$ exists and depends on the class $[F]$ only. 

Given a $z>0$, there is a $\bC$-linear map $Z_B^\cX: R^\vee_q \to \bC$. It extends to a multi-valued holomorphic family of linear maps $R^\vee_q\to \bC$ over $q\in \cM$. Or one can regard it as a flat section of $\tcR$.

In Section \ref{sec:fourier} we show that for $\tcL=\cO(\sum_{i=1}^r c_i \tcD_i)$ and $t$ in a certain strip compatible with $\tcL$,
\[
    Z^{\cX,G}_B(F_{\tcL'}\cap  \bigcap_{a=1}^k \{\Re\tH_a=\Re(t_a)\}) = Z_A^{\cX}(\cL)(t).
\]
In this section, we explore what happens for more general $t$.

Following \cite[Section 3.1.2]{Iritani_2009} we define a splitting of the following exact sequence for the extended Picard group over $\bQ$
\begin{equation}\label{eq:extra_picard_exact_sequence}
   0\rightarrow \bigoplus_{i=r'+1}^r \bQ D_i \rightarrow \bL^\vee_\bQ \rightarrow H^2(\cX,\bQ) \rightarrow 0.
\end{equation}
For each $j\in\{r'+1,\dots,r\}$, suppose $b_j$ is contained in the cone generated by $I_j\subset\{1,\dots,r'\}$, i.e., $b_{j} = \sum_{i\in I_j} s_{ji} b_i$ for some $s_{ji} > 0$. We choose $D_j^\vee\in \bL \otimes \bQ$ such that
\begin{equation*}
    \langle D_i,D_j^\vee \rangle = \left\{\begin{array}{cc}
        1 & i=j  \\
        -s_{ji} & i \in I_j\\
        0 & i \notin I_j\cup\{j\}
    \end{array}\right.
\end{equation*}
Note that $\sum_{i\in I_j} s_{ji} \leq 1$ due to Assumption \ref{assump:positive}. We then call $\bL^\vee_\bQ = (\Ker(D_{r'+1}^\vee,\dots,D_r^\vee)) \oplus \bigoplus_{i=r'+1}^r \bQ D_i$ the canonical splitting of (\ref{eq:extra_picard_exact_sequence}), and $\iota:\Ker(D_{r'+1}^\vee,\dots, D^\vee_r)\cong H^2(\cX;\bQ)\hookrightarrow \bL_\bQ^\vee$ is canonically identified as a subspace of $\bL^\vee_\bQ$. This splitting induces a decomposition $\tcM=\tcM_\Kah \times \tcM_\orbifold$ and $\cM=\cM_\Kah \times \cM_\orbifold,$ where 
\[\tcM = \bL^\vee_{\bC}, \; \tcM_\Kah = \Ker(D_{r'+1}^\vee,\dots,D_r^\vee), \; \tcM_\orbifold = \bigoplus_{i=r'+1}^r \bC D_i.\]

In addition, we choose a splitting $\ell$ of the sequence \eqref{eqn:MtM} over $\bQ$
\begin{equation} \label{eq:ell}
0\longrightarrow M_\bQ\longrightarrow \tM_\bQ \stackrel{\stackrel{\ell}{\leftarrow}}{\longrightarrow} \bL^\vee_\bQ \longrightarrow  0. 
\end{equation}
Under the splitting $\bL^\vee_\bQ= \Ker(D_{r'+1}^\vee,\dots,D_r^\vee) \oplus \bigoplus_{i=r'+1}^r \bQ D_i$ and $\tM_\bQ=\bigoplus_{i=1}^{r'} \bQ D_i^\bT \oplus \bigoplus_{i={r'+1}}^r \bQ D_i^\bT$, we require
\begin{gather*}
\ell(D_i)=D_i^\bT,\ i=r'+1,\dots,r.
\end{gather*}

Given an equivariant line bundle $\cL^\sharp=\cO(\sum_{i=1}^{r'} c_i \cD_i) = \sum_{i=1}^{r'}c_i \bar{D}^{\bT}_i$, it corresponds to a non-equivariant line bundle $\cL=\sum_{a=1}^{k'} {h_a} \bar p_a\in \Pic(\cX)$ as in \eqref{eq:ha}, where $\bar p_a$ are the images of $p_a$ under $\bL^\vee \to \Pic(\cX)$ for $1<a<k'$.

We have the following commutative diagram:
\begin{equation}\label{eqn:diagram}
\begin{CD}
&&  && 0 && 0 \\
& & @.  @VVV  @VVV \\
 @. 0 @>>> \bigoplus_{i=r'+1}^{r} \bQ D_i^{\bT} @>{\cong}>> \bigoplus_{i= r'+1}^{r} \bQ D_i @>>> 0\\
& & @VVV  @VVV  @VVV \\
0 @>>> M_\bQ  @>{\phi^\vee}>> \tM_\bQ  @>{\psi^\vee}>{\stackrel{\longleftarrow}{\ell}}> \bL^\vee_\bQ @>>> 0\\
    & & @VV{\cong}V  @VV{\nu_\bT}V  @V{\iota\uparrow}V{\nu}V & \\
0 @>>> M_\bQ  @>{\bar{\phi}^\vee}>> H^2_{\bT}(\cX;\bQ)  @>{\bar{\psi}^\vee}>> H^2(\cX;\bQ) @>>> 0\\\
&&  @VVV @VVV @VVV & \\
&& 0 && 0 && 0 &  \\
\end{CD}
\end{equation}


We start by constructing cycles in $M_\bR \times N_\bR$. They will be the Lagrangian cycles produced by the \emph{coherent-constructible correspondence} (CCC) \cite{bondal_2006, F-L-T-Z_2011,Kuwagaki2020}, a version of homological mirror symmetry for toric orbifolds.

For $c=(c_1,\dots, c_{r'})\in \bR^{r'}$, we define the conical piecewise-linear Lagrangian set
\[
\widetilde\Lambda_c10=\bigcup_{\sigma \in \Si} (-\si) \times \si^\perp_{\bZ,c} \subset N_\bR\times M_\bR.
\]
where $\si^\perp_{\bZ,c}=\{\chi\in M_\bR: \langle\chi,b_i \rangle\in c_i+\bZ,\ \forall \rho_i \in \si\}.$ The conical Lagrangian set $\Lambda_c\subset N_\bR\times M_\bR/M$ is the image of $\widetilde \Lambda_c$ under the covering map $M_\bR\to M_\bR/M$. We denote $\widetilde\Lambda=\widetilde\Lambda_0$ and $\Lambda=\Lambda_0$.

The CCC states that there are exact quasi-equivalences of dg categories for $\bT$-equivariant and non-equivariant coherent sheaves on $\cX$
\begin{align*}
\kappa_\bT:\ &\Coh_{\bT}(\cX)\xrightarrow{\sim} \Sh_{c,\widetilde{\Lambda}}(M_\bR),\\
\kappa:\ &\Coh(\cX)\xrightarrow{\sim} \Sh_{\Lambda}(M_\bR/M).
\end{align*}
The category $\Sh_{c,\widetilde{\Lambda}}(M_\bR)$ is the dg category of $\bC$-module sheaves with constructible and compactly-supported cohomology whose \emph{singular support} is in $\widetilde\Lambda$. Similarly the $\Sh_{\Lambda}(M_\bR)$ is the dg category of $\bC$-module sheaves with constructible cohomology whose singular support is in $\Lambda$. The \emph{characteristic cycle} map $\CC$ (\cite[Chapter IX]{kashiwara_schapira}), which associates a conical Lagrangian cycle to a constructible sheaf. We denote 
\[
\tS(-c)=\CC(\kappa_\bT(\cL^\sharp)),
\]
where $\cL^\sharp=\cO(\sum_{i=1}^{r'} c_i \cD_i)$ for $c_i\in \bZ$ is an equivariant line bundle on $\cX$. 

To compute $\tS(-c)$ explicitly, consider the construction of $\kappa_\bT$ as in \cite[Corollary 3.5]{F-L-T-Z_2011} and \cite[Theorem 5.14]{FLTZ14}. Let $\si_I$ be the cone spanned by $\rho_i$ for any subset $I\subset \{1,\dots, r'\}$ if such cone $\si_I\in \Si$.


The following resolution of $\cL^\sharp$ 
\[
\bigoplus_{\dim \si_I=n} \iota_{\si_I*}\cL^\sharp\vert_{\cX_{\si_I}} \to \bigoplus_{\dim \si_I=n-1} \iota_{\si_I*} \cL^\sharp\vert_{\cX_{\si_I}} \to \dots
\]
is quasi-equivalent to $\cL^\sharp$. Here $\cX_{\si_I}$ is the affine toric orbifold given by the cone $\si_I$ and $\iota_{\si_I}$ is its inclusion into $\cX$. 

For any $I$ such that $\si_I\in\Si$, let $C^\vee_I=\{m\in M_\bR, \langle m, \si_I\rangle \geq 0 \}$ be the dual cone of $\si_I$. The equivariant line bundle $\cL^\sharp$ restricted to $\cX_{\si_I}$ is just the structure sheaf on $\cX_{\si_I}$ twisted by a character $\chi_I = \chi_I(c)\in M_\bQ/ \si_I^\perp$. Here $\chi_I(c)$ is a linear function of $c=(c_1,\dots,c_{r'})$, characterized by $\langle \chi_I,b_i \rangle = -  c_i$ for each $1$-cone $\rho_i \in \si_I$.

Let $C_I^{\chi_I,\vee}=\chi_I + C_I^\vee\subset M_\bR$. Relaxing to quasi-coherent sheaves (while the images are no longer compactly supported), the CCC functor $\kappa_\bT$ is explicitly given by
\[
\kappa_\bT( \iota_{\si_{I}*} \cL^\sharp\vert_{\cX_{\si_I}} ) =j_{(C^{\chi_I,\vee}_I)^\circ !}(\omega_{(C^{\chi_I,\vee}_{I})^\circ}),
\] 
where $j_{(C^{\chi_I,\vee}_I)^\circ}:(C^{\chi_I,\vee}_I)^\circ\hookrightarrow M_\bR$ is the inclusion and $\omega_{(C^{\chi_I,\vee}_{I})^\circ}$ is the dualizing sheaf on $(C^{\chi_I,\vee}_{I})^\circ$, isomorphic to the shifted constant sheaf $\bC_{(C^{\chi_I,\vee}_{I})^\circ}[-\dim(\cX)]$. Notice that 
\[
\CC(j_{(C^{\chi_I,\vee}_I)^\circ !}(\omega_{(C^{\chi_I,\vee}_{I})^\circ}))= \tS_I(-c):= \sum_{\text{cone } C \subset \si_I}(\chi_I+C^\perp \cap C^{\vee}_I)\times (-C) \subset M_\bR \times N_\bR.
\]

Therefore, we define $\tS(-c)$ as the following for $c=(c_1,\dots, c_{r'})$ where $c_i\in \bR$ simply by extending the linear function $\chi_I(c)$ to non-integer $c$.
\begin{definition} 
\begin{equation}
\label{eqn:CCC-cycle-definition}
\tS(-c)=\sum_{\substack{I\subset\{1,\dots,r'\},\\ \text{$\sigma_I$ is a cone}}} (-1)^{|I|-n}\tS_I(-c).
\end{equation}
\label{def:CCC-cycle-definition}
\end{definition}

It is obvious that $\tS(-c)$ has no boundary since  $\tS_I(-c)$ has no boundary. And $\tS(-c)$ is horizontally compactly supported since $\kappa_\bT$ maps coherent sheaves, including line bundles, to compactly supported constructible sheaves. 

We call such cycle $\SYZ(\cL^\sharp):=\CC(\kappa_\bT(\cL^\sharp))=\tS(-c)$ the SYZ mirror cycle of $\cL^\sharp,$ following the mirror symmetry principle established in \cite{SYZ1996}.


\begin{remark}
When the $\bT$-equivariant line bundle $\cL^\sharp$ is \emph{ample}, its equivariant moment polytope is
\[ 
\Delta_c=\{m\in M_\bR: \langle m, b_i\rangle \geq -c_i, i=1,\dots, r'\}.
\]
Under the CCC functor 
\[
\kappa_{\bT}(\cL^\sharp)=j_{\Delta^\circ_c!} \omega_{\Delta^\circ_c},
\]
where $j_{\Delta_c^\circ}:\Delta^\circ_c\hookrightarrow M_\bR$ is the embedding of the open subset $\Delta^\circ_c$, and $\omega_{\Delta_c^\circ}$ is the dualizing sheaf. We have 
\[
\CC(j_{\Delta^\circ_c!} \omega_{\Delta^\circ_c})=\tS(-c)=\sum_{\text {$F$ is face of $\Delta_c$}} F \times (-C_F) .
\]
Here $C_F$ is cone dual to $F$. 
\end{remark}

We give another description for $\tS(-c)$. For each $c = (c_1,\dots,c_{r'})\in\bR^{r'}$ and $i\in\{1,\dots,r'\}$, let $H_i(c) = \{m\in M_\bR\mid \langle m,b_i \rangle = -c_i\}$ and $V_i^{\pm}(c) = \{m\in M_\bR\mid \langle m,b_i \rangle \geq(<) -c_i\}$ be the hyperplanes and closed(open) halfspaces in $M_\bR$. 

The hyperplane arrangement $\{H_i(c)\}_{i=1}^{r'}$ provides a cellular decomposition of $M_{\bR}$ into pieces.

For each $I \subset \{1,\dots,r'\}$, set
\[
    H_{I}(c) = \bigcap_{i \in I} H_i(c), V_{I}(c)^\pm = \bigcap_{i \in I} V_i^\pm(c).
\]
We set $H_{\emptyset} = V_{\emptyset}^{\pm} = M_\bR$. For each $I \subset \{1,\dots,r'\}$ and $J \subset \{1,\dots,r'\}\backslash I$, let $\mathrm{d}_{I,J} = \dim(H_{I}(c) \cap V_{J}^+(c) \cap V^-_{\{1,\dots,r'\} \backslash (I \cup J) }(c))$, we consider the subset

$$
    \tau_{I,J}(c) = 
    \begin{cases}
        \overline{H_{I}(c) \cap V_{J}^+(c) \cap V^-_{\{1,\dots,r'\} \backslash (I \cup J) }(c)}, &  \mathrm{d_{I,J}} + \dim(\Span_{\bR}\{b_i\mid i \in I\}) = n,\\
         \emptyset, & \mathrm{d_{I,J}} + \dim(\Span_{\bR}\{b_i\mid i \in I\}) < n.
    \end{cases}
$$
We call it a cell of $H_{I}(c)$ whenever it is not empty. We say that $I$ generates a cone if $\{b_s, s\in I\}$ generates a cone of $\Sigma$ and denote the cone by $\sigma_{I}$. We say $I=\emptyset$ generates the cone which is the origin of $N_{\bR}$, and set $\sigma_{\emptyset}$ to be the origin. For each cell $\tau_{I,J}(c)$, set
\begin{align*}
    \mathfrak{D}_k(I,J) = &\{K \subset \{1,\dots,r'\}  \backslash I \mid |I| + |K| + k = n, K\subset J, \\
    & I\amalg K\text{ generates a cone}\},\\
    d_k(I,J) = & |\mathfrak{D}_k(I,J)|.
\end{align*}
In other words, $d_k(I,J)$ is a number all the possible ways to complete $\sigma_I$ to a codimension $k$ cone $\sigma_{I\sqcup K}$ such that the cell
$\tau_{I,J}(c)$ is contained in $V^+_K(c)$.
We call 
\begin{equation} \label{eq:multiplicity}
    m_{I,J} = \sum_{k=0}^{r'} (-1)^k d_{k}(I,J)
\end{equation}
the \textit{multiplicity} of $\tau_{I,J}(c)$. Note that $m_{I,J}$ is independent of $c$.

\begin{proposition}
With $c \in \bR^{r'}$, the cycle $\tS(-c)$ in $M_{\bR} \times N_{\bR}$ is by the formula
\[
    \tS(-c) = \sum_{\{I \mid I \text{ generates a cone}\}, I \cap J = \emptyset} m_{I,J} (\tau_{I,J}(c) \times (-\sigma_{I})),
\]
where we set $\tau_{I,J}(c) \times (-\sigma_{I}) = 0$ if $\tau_{I,J}(c) = \emptyset$.
\end{proposition}
\begin{proof}
From Definition \ref{def:CCC-cycle-definition}, one just needs to count the signed number of times that $\tau_{I,J}(c)\times(-\sigma_I)$ appears in $\tS_L(-c)$, where $L$ ranges over all cones $\sigma_L$ in $\Si$. This happens whenever $L = I\sqcup K$ for some $K\subset J$. Notice that $m_{I,J}$ is the signed number of times $\sigma_{I\sqcup K}$ is a cone of $\Sigma$ for all possible $K \subset J$. Therefore $m_{I,J}$ is the signed counting of all cones $\si_L$ such that $\tau_{I,J}(c)\times(-\sigma_I) \subset \tS_L(-c)$.
\end{proof}

\begin{example}
    Consider $\cX = \bP^2$ with $b_1 = (1,0), \; b_2 = (0,1), \; b_3 = (-1,-1)$. We have
    \[
        m_{I,J} = 
        \begin{cases}
        0 & \text{if } (|I|,|J|) = (0,1),(0,2),(1,1),(3,0)\\
        1   & \text{if } (|I|,|J|) = (0,0),(0,3),(1,2),(2,0),(2,1)\\
        -1 & \text{if } (|I|,|J|) = (1,0)
        \end{cases}
    \]
    Choose $c = (c_1,c_2,c_3)$. We present a few cells $\tau_{I,J}$ with multiplicities $m_{I,J}$ for the cases $c_1+c_2+c_3>0$ and $c_1+c_2+c_3<0$ in Figure \ref{fig:cells}.
    \begin{figure}[ht]
    \centering
    \begin{subfigure}{0.45\textwidth}
        \centering
        \def\svgwidth{\textwidth}
        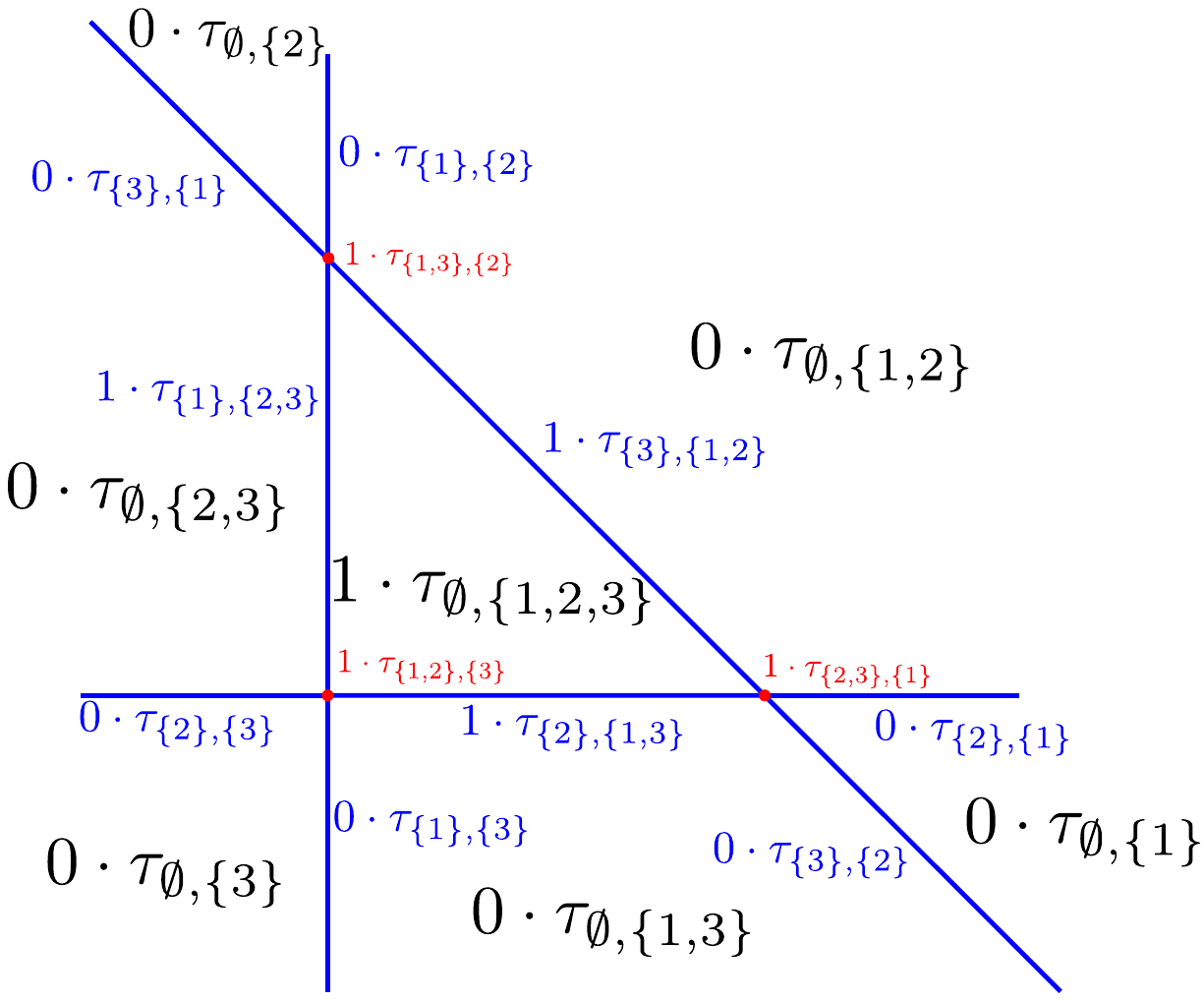
        \caption{$c_1+c_2+c_3>0$}
        \label{fig:cells_c>0}
    \end{subfigure}
    \hfill
    \begin{subfigure}{0.45\textwidth}
        \centering
        \def\svgwidth{\textwidth}
        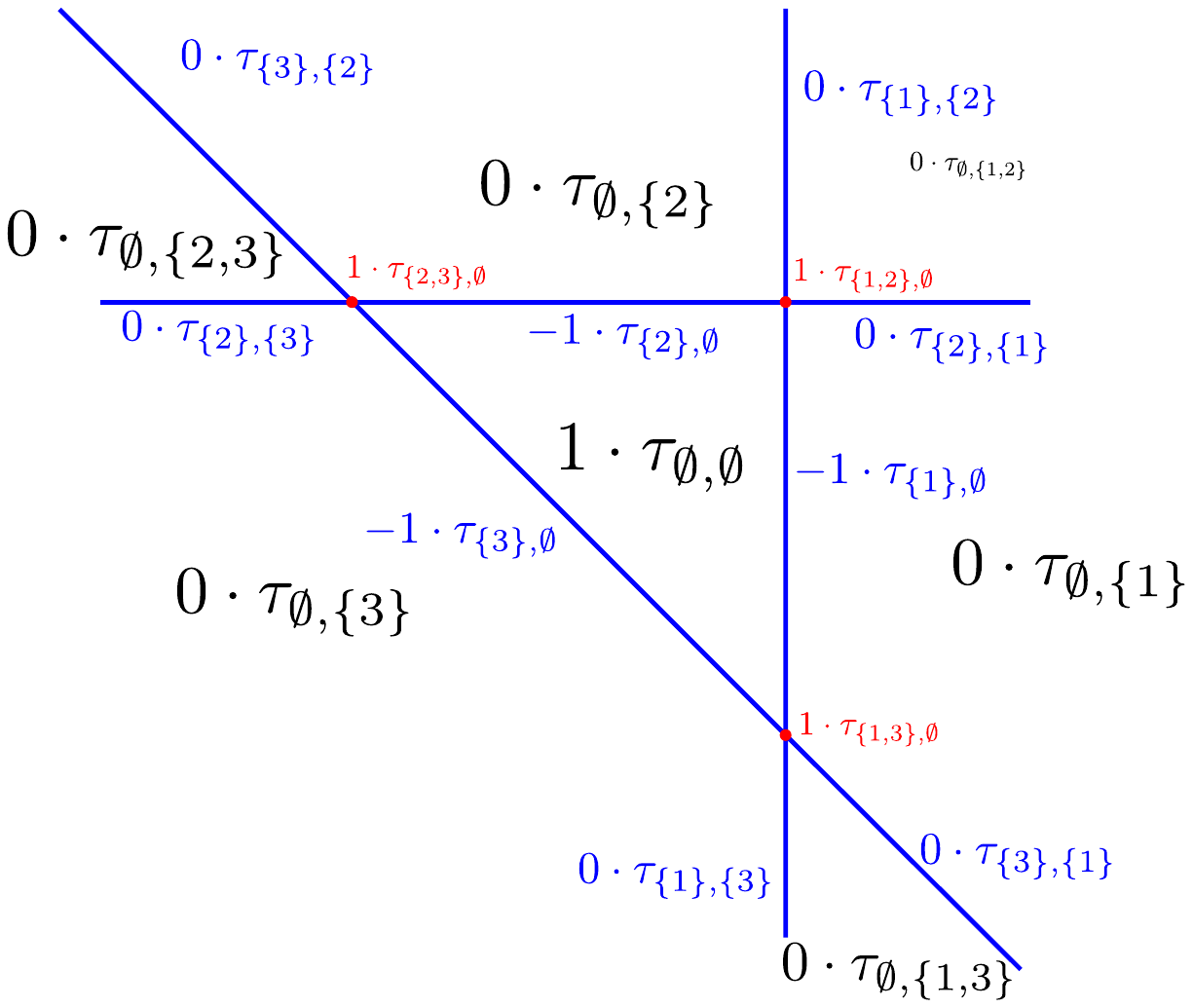
        \caption{$c_1+c_2+c_3<0$}
        \label{fig:cells_c<0}
    \end{subfigure}
    \caption{Cells $\tau_{I,J}$ for $\cX = \bP^2$}
    \label{fig:cells}
  \end{figure}

\end{example}

We then construct a bijection $\rho:N_{\bR} \rightarrow M_\bR$. Set
\begin{align*}
    P &= \conv\{b_i\mid i\in\{1,\dots,r'\}\} \subset N_\bR, \\
    \beta_i(m) &= \langle m, b_i\rangle,\\
    U_I &= \{ m \in M_\bR \mid \max_{i\in I}\beta_i(m) > \beta_j(m), \forall j \in \{1,\dots,r'\} \backslash I\} \subset M_\bR.
\end{align*}

\begin{lemma}\label{lemma:iso_from_N_to_M}
    There exists a conical homeomorphism $\rho:N_\bR\rightarrow M_\bR$ such that for any $I\subset\{1,\dots,r'\}$ which generates a cone $\sigma_I$, we have 
    \[
        \rho(-\sigma_I) \subset U_I.
    \]
\end{lemma}

\begin{proof}
    Since the fan $\Sigma$ is Fano, $F_I := \sigma_I \cap \partial P$ is a face of $P$ whose vertices are $\{b_i \mid i\in I\}$, and the dual poyltope is given by 
    \[
        P^\vee = \{m \in M_\bR \mid \langle m,n \rangle \leq 1, \forall n \in P \}.
    \]
    Then $F_I^\vee := \{m\in M_\bR\mid \langle m,n_i \rangle = 1,\forall i\in I\} \cap \partial P^\vee$ is a face of $P^\vee$. 
     
    Let $\sigma_I^\vee$ be the cone over $F_I^\vee$, then
    \[
        U_I = (\bigcup_{I'\subset I}\sigma_{I'}^\vee)^\circ.
    \]
    Consider the barycentric subdivisions $\mathfrak{B}$ and $\mathfrak{B}^\vee$ of $\partial P$. Let $b_I$, $b_I^\vee$ be the barycenters of $F_I$, $F_I^\vee$. For any $\tilde{n}\in \partial P$, let $\Delta = \conv\{b_{I_1},\dots,b_{I_s}\}$ be the minimal simplex of $\mathfrak{B}$ which contains $\tilde{n}$. Then $\tilde{n} = \sum_{j=1}^s \mu_j b_{I_j}$ for some $\mu_j\in(0,1)$ with $\sum_{j=1}^s\mu_j = 1$. Let us set
    \[
    \tilde{\rho}(\tilde{n}) = \sum_{j=1}^s\mu_jb^\vee_{I_j}\in\partial P^\vee.
    \]
    Then for any $n\in N_\bR$, let $\tilde{n} = \R_{>0}\cdot n \cap\partial P$ and $n = l \tilde{n}$ for some $l>0$. Let us set
    \[
        \rho(n) = -l\tilde{\rho}(\tilde{n}).
    \]
    Then we have $\rho(-\sigma_I)\subset U_I$.
\end{proof}

\color{black}

Now let us define the following isomorphism
\begin{align*}
\si_t: M_\bR \times M_\bR =(\widetilde{\bC^*})^n &\stackrel{\sim}{\rightarrow}  \tcY_t,\\
(\log Z_1,\dots, \log Z_n) &\mapsto (\sum_{a=1}^k t_a{\ell_{1a}} + \sum_{j=1}^n b_{1j} \log Z_j,\dots, \sum_{a=1}^k t_a {\ell_{ra}}+  \sum_{j=1}^n b_{rj}\log Z_j)
\end{align*}
for $t = (t_1,\dots,t_k) \in \tcM$. Then we set
\[
    P_t:M_\bR\times N_\bR \rightarrow\tcY_t
\]
to be $P_t = \sigma_t \circ(\id_M \times \rho)$. Let $\pi_M:M_\bR\rightarrow M_\bR/M$ be the quotient map.

For $c=(c_1,\dots,c_{r'})\in \tM$, let $\cL^\sharp=\cO(\sum_{i=1}^{r'} c_{i} \cD_i)$ be a $\bT$-equivariant line bundle on $\cX$. Let $c_j=0$ for $j>r'$, and let $h_a=2\pi \sum_{i=1}^r l^{(a)}_i c_i$. 


\begin{definition} \label{def:SYZ_cycle}
    We define the following family of cycles for $t\in \tcM$
    \[
        \fP_c(t) = P_t\left(\tS\left(\ell'\left(\Im(t)/2\pi\right)-c\right)\right),
    \]
    where $\ell'=\nu_\bR\circ \ell$ and $\ell$ is the splitting~\eqref{eq:ell}.    
\end{definition}

\begin{proposition} \label{prop:cycles_are_good}
    There is a small neighbourhood $|q|<\epsilon$ in $\cM^\circ$ such that with $q=e^t$, $\fP_c(t)$ represents an element in $H_n(\cY_{e^t};\realpart(W^\cX) \gg 0)$.
\end{proposition}

\begin{proof}
    We have $P_t(m,n) = \left( \alpha^t_1(n)+\sqrt{-1}\beta^t_1(m), \dots, \alpha^t_r(n)+\sqrt{-1}\beta^t_r(m) \right)$ with 
    \[
        \alpha^t_i(n) = \sum_{a=1}^k \Re(t_a) \ell_{ia} + \langle \rho(n), b_i\rangle, \beta^t_i(m) = \sum_{a=1}^k \Im(t_a)\ell_{ia} + 2\pi\langle m, b_i\rangle.
    \]
    Recall that $$\tS(c) = \sum_{I,J\in [r'], \si_I\in \Si, I \cap J = \emptyset} m_{I,J} (\tau_{I,J}(c) \times (-\sigma_I)).$$
    It is sufficient to show that 
    $$
        P_t\left(\tau_{I,J}\left(\ell'\left(\Im(t)/2\pi\right)-c\right) \times (-\sigma_I)\right)
    $$
    represents an cell in the relative chain group $C_n(\cY_{e^t};\realpart(W^\cX) \gg 0)$ for any $I,J$. For any $i \in I$, we have 
    \[
        \frac{\beta^t_i(m)}{2\pi} = c_i \in \bZ
    \]
    since $m \in \tau_{I,J}(\ell'(\frac{\Im(t)}{2\pi})-c))$. 
    \begin{multline*}
        W(P_t(m,n)) = \sum_{i \in I} e^{\alpha_i^t(n)} + \sum_{j \in \{1,\dots,r'\}\backslash I} e^{\alpha_j^t(n) + \sqrt{-1}\beta_j^t(m)} + \sum_{\{k\in\{r'+1,\dots,r\} \}} e^{\alpha_k^t(n) + \sqrt{-1}\beta_k^t(m)}
    \end{multline*}
    for $(m,n) \in \tau_{I,J}(\ell'(\frac{\Im(t)}{2\pi})-c) \times (-\sigma_I)$. For $k\in\{r'+1,\dots,r\}$, notice that 
    \[
    \alpha_k^t(n) = \Re((t_{\orbifold})_k)(\sum_{i\in I_k} s_{ki}\alpha_i^t(n)),
    \] 
    where $0<s_{ki}<1, \sum_{i\in I_k} s_{ki} \leq 1$ due to Assumption \ref{assump:positive} and we write $t = (t_{\Kah},t_\orbifold)$ according to the decomposition $\tcM=\tcM_\Kah \times \tcM_\orbifold$. We choose a neighbourhood of $|q|<\epsilon$ in $\cM^\circ$ and some $C>0$ such that for any $e^t$ in the neighbourhood, $\Re (t_\orbifold)$ is sufficiently negative so that
    \[
    \max_{i\in [r']}\{\alpha_i^t(n)\} > \alpha_k^t(n) + C
    \]
    for $n\in N_\bR$ outside of a compact subset. Furthermore, there exists a compact subset $K\subset N_\bR$ such that 
    \begin{align*}
        \max_{i\in I}\{\alpha_i^t(n)\} >& \alpha_j^t(n) + C, \forall j\in\{1,\dots,r'\}\backslash I,\\
        \max_{i\in I}\{\alpha_i^t(n)\} >& \alpha_s^t(n) + C, \forall s\in\{r'+1,\dots,r\}
    \end{align*}
    for any $n \in -\sigma_I\backslash K$ due to Lemma \ref{lemma:iso_from_N_to_M}. This implies $\fP_c(t)$ represents an element in $H_n(\cY_{e^t};\realpart(W^\cX) \gg 0)$.
\end{proof}

When $t$ is real, $\fP_c(t)$ is the SYZ mirror cycle $\si_t\circ(\id_{M_\bR}\times \rho)(\mathrm{SYZ}(\cL^\sharp))$ up to a rational shift in $M_\bR$ and the identification $\rho$ between $M_\bR$ and $N_\bR$.

\begin{lemma}
    The cycles $\fP_c(t)$ in Proposition \ref{prop:cycles_are_good} represent flat sections of the local system $\widetilde{\cR}^\vee$ on $\widetilde{U}_\epsilon = \{t\mid |e^t| < \epsilon\}$.
\end{lemma}

\begin{proof}
    Let $\sigma_{t_1,t_2}:= (\sigma_{t_2} \circ \sigma_{t_1}^{-1}:\tcY_{t_1}\rightarrow \tcY_{t_2})$ be the isomorphism from $\tcY_{t_1} \cong \cY_{e^{t_1}}$ to $\tcY_{t_2}\cong \cY_{e^{t_2}}$. It is sufficient to show that given $t,t'\in\widetilde{U}_\epsilon$, if $|t-t'|$ is small enough, then $\sigma_{t',t}(\fP_c(t'))$ and $\fP_c(t)$ represent the same element in $H_n(\cY_{e^t};\realpart(W^\cX) \gg 0)$. 
    
    From the proof of Proposition \ref{prop:cycles_are_good}, there exists some $\delta > 0$ such that whenever $t'\in B_\delta(t) =\{t''\mid |t''-t|<\delta \}\subset \widetilde{U}_\epsilon$, $\sigma_{t',t}(\fP_c(t'))$ represents an element in $H_n(\cY_{e^t};\realpart(W^\cX) \gg 0)$. Now let us choose a continuous path $p:[0,1]\rightarrow B_\delta(t)$ with $ p(0) = t, p(1) = t'$. For each $s\in[0,1]$, write $a(s) = \ell'\left(\Im(p(s))/2\pi\right)-c$ in Definition \ref{def:SYZ_cycle} so that $\fP_c(p(s)) = P_{p(s)}(\tS(a(s)))$.
    We can regard 
   \[
        \widetilde{S}:= \bigcup_{s\in[0,1]}\tS(a(s))
   \]
   as an $(n+1)$-chain in $M_\bR \times N_\bR$. Recall that $P_t = \sigma_t\circ(\id_M \times \rho):M_\bR\times N_\bR\rightarrow \cY_{e^t}$, so $\sigma_{p(s),p(0)}\circ P_{p(s)} = \sigma_{p(0)}\circ(\id_M \times \rho)$, which is independent of $s$. Let us consider the $(n+1)$-chain 
   \[
        \widetilde{\fP} = \big(\sigma_{p(0)}\circ(\id_M \times \rho)\big)_*(\widetilde{S})
   \]
    in $\cY_{e^t}$. Then $\widetilde{\fP}$ is an element in $C_{n+1}(\cY_{e^{t}};\realpart(W^\cX) \gg 0)$ and $\partial \widetilde{\fP} = \sigma_{t',t}(\fP_c(t')) - \widetilde{\fP}_c(t)$. So $\sigma_{t',t}(\fP_c(t'))$ and $\fP_c(t)$ represent the same element in $H_n(\cY_{e^t};\realpart(W^\cX) \gg 0)$ as wanted.
\end{proof}

\begin{theorem}
    Suppose $\cL^\sharp=\cO(\sum_{i=1}^{r'} c_i\cD_i)\in \Pic_\bT(\cX)$. It descends to a line bundle $\cL\in \Pic(\cX)$. For any $t \in \tcM$ such that $|e^t|$ is sufficiently small, we have
    \begin{equation} \label{eq:centralChagesMirror2}
        Z_A^{\cX}(\cL)(t)= (2\pi\sqrt{-1})^{-n}Z^{\cX}_B(\fP_c(t)) .
    \end{equation}
    \label{thm:main-2}
\end{theorem}

\begin{proof}
    We construct a rational Picard group element $\tcL=\cO(\sum_{i=1}^r c_i \cD_i)=\iota_\bT(\cL^\sharp)\in \Pic_{\tbT}(\tM_\bC)$. From the definition when $\Im (t_a)=2\pi h_a, \; a\in 1,\dots,k$, i.e., when $t$ is chosen compatible with $\tcL$, $\fP_c(t) = F_\tcL\cap  \bigcap_{a=1}^k \{\Re(\tH_a)=\Re(t_a)\}$.
    
    Then, the formula~\eqref{eq:centralChagesMirror2} follows from the Theorem \ref{thm:main-1} for $t$ compatible with $\tcL$. 
    Thus, both sides of the equation~\eqref{eq:centralChagesMirror2} are analytic functions of $t$ coinciding on a real subspace $\Im(t_a) = 2\pi h_a, a\in 1,\dots,k$,
    that are defined for all $\Im(t)$, so they are equal. 
\end{proof}

\appendix

\section{Convergence of inverse Fourier transform}

We recall some basic facts about Fourier transforms of analytic functions and how they apply to this paper.
The following theorem is a standard fact in the theory of Fourier transforms.

\begin{theorem}[Fourier inversion theorem for analytic functions] \label{thm:fourierInversion}
Let $f(x) \; : \; \bC \to \bC$ be a meromorphic function which is holomorphic in the strip $\bR \times \sqrt{-1}[a,b]$ and satisfying
    \begin{equation}
        f(x) = 
        \begin{cases}
            O(e^{-(\beta-\epsilon)\Re(x)}, \quad \Re(x) \to \infty, \\
            O(e^{-(\alpha-\epsilon)\Re(x)}, \quad \Re(x) \to -\infty,
        \end{cases}
    \end{equation}
    for $a<b$ and $\beta > \alpha$
    uniformly in $\Im(x) \in [a,b]$. Then, its Fourier transform is defined as 
    \[
        g(u) := \int_{\bR+\sqrt{-1}c}e^{ux}f(x)dx, \quad c \in [a,b].
    \]
    It is an analytic holomorphic function in a strip $u \in [\alpha,\beta]\times \sqrt{-1}\bR$ such that 
    \begin{equation}
        g(u) = 
        \begin{cases}
            O(e^{-(b-\epsilon)\Im(u)}, \quad \Im(u) \to \infty, \\
            O(e^{-(a-\epsilon)\Im(u)}, \quad \Im(u) \to -\infty,
        \end{cases}
    \end{equation}
    uniformly in $\Re(u) \in [\alpha,\beta]$.
    Moreover,
    the inverse Fourier transform of $g(u)$ is 
    \[
        f(t) = \InvM(g)(t) = (2\pi\sqrt{-1})^{-1}\int_{\gamma+\sqrt{-1}\bR} e^{-ut} g(u)du, \quad \gamma \in [\alpha,\beta].
    \] 
\end{theorem}
For example, consider the function $f(x) = e^{-e^x}$. This function satisfies the conditions of the theorem with $b=\pi/2-\delta, \; a = -b, \; \alpha = 0, \beta =
\infty$. The Fourier transform of $f(x)$ is the Gamma function $\Gamma(u)$ whose inverse Fourier transform is again $f(x)$ by the residue theorem.

\paragraph{\bf Convergence of partial inverse Fourier transforms}

We also prove the following convergence lemma needed to generalize the argument of~\cite{Aleshkin_Liu_23}.
First, we recall the definitions of an equivariant deformation and cylinder-like integration contours.
Consider an ``equivariant deformation'' $\alpha = (\alpha_1, \ldots, \alpha_r) \in \bR^r$ and define the hyperplanes 
$\pi^i_{m_i} \subset \fg$ by the equation
\[
    \pi^i_{m_i} = \{\sum_a u_a l^{(a)}_i + \alpha_a = -m_i\}, \quad m_i \in \bZ_{\ge 0}.
\]
Gamma function $\Gamma(\sum_a u_a l^{(a)}_i +\alpha_i)$ has a pole at $\pi^i_{m_i}$ for each $m_i \ge 0$.
Let $I \subset [1..r]$ be an anticone
and $\delta \in \bigcap_{i\in I} \pi^i_{m_i}$, where $m_i \in \bZ_{\ge 0}, \; i \in I$.

Cylinder-like integration contours are crucial in the argument of~\cite{Aleshkin_Liu_23}. They appear in the process of contour deformations when the contour intersects polar hyperplanes of gamma functions. We need to have an estimate for integrals over such
contours in order to control the convergence of the central charges.

We can define the cylinder contour $\cC(\delta) \subset \mathrm{Lie}(G)$ centered at $\delta$ as follows.
Topologically, $\cC^I(\delta)$ is the product of $|I|$ circles and $k-|I|$ real lines. Geometrically,
let $\fg_I \subset \fg$ be the Lie subalgebra defined by the equations $\fg_I = \bigcap_{i\in I}\{\sum_a u_a l^{(a)}=0\}$ and
$\mathfrak{h}_I = \fg/\fg_I$. The hyperplanes $\pi^i_{m_i}$ project to hyperplanes $\overline{\pi}^i_{m_i} \subset \mathfrak{h}_I$ and
the divisor $\bigcup_{i\in I} \overline{\pi}^i_{m_I}$ is a simple normal crossing divisor with the intersection $\overline{\delta}$ 
which is the image of $\delta$ under the projection map. We let $\overline{\cC}(\overline{\delta})\simeq (S^1)^{|I|}$ to be a small torus in 
$\mathfrak{h}_I$ encircling the normal crossing divisor. Consider $\cC(\delta)$ to be a lift of $\overline{\cC}(\overline{\delta})$
to $\fg$ contained in a small ball around $\delta$. Finally, we define $\cC(\delta)\simeq(S^1)^{|I|}\times\bR^{k-|I|}$ to be $\cC(\delta)'+\sqrt{-1}(\fg_I)_\bR$. We will not discuss the choice of an orientation here as it is not relevant to the convergence 
discussion. 

\begin{lemma} \label{lem:convergence}
    Let $\delta$ be in the effective cone, and $\cC_{\delta}$ be as above. Then, there exist $C_1, \; C_2 > 0$ and $\alpha$ small enough such that we have the following estimate:
    \begin{equation} \label{eq:intEstimate1}
        I := \int_{\cC(\delta)}\left|\prod_{a=1}^k du_a e^{\sum_a t_a u_a} \prod_{i=1}^r \Gamma\left(\sum_a u_a l^{(a)}+\alpha_i\right)\right|
        < C_1 e^{(\sum_i \delta_i \ln|\delta|)+\sum_a t_a\delta_a +C_2|\delta|},
    \end{equation}
    where $|\cdot|$ is a chosen norm on $\fg$.
\end{lemma}
\begin{proof}
 Let $z_i = x_i + \sqrt{-1}y_i= \sum_a u_a l^{(a)}+\alpha_i$. We use lemma A.1 from~\cite{Aleshkin_Liu_23} to estimate the integrand:
 \begin{multline}
     |\prod_{i=1}^r \Gamma(z_i)| < C_1\prod_i e^{x_i \ln|z_i|}e^{-\frac{\pi}{2}|y_i|}e^{C_2 |x|} = C_1\prod_i e^{x_i \ln|z|}e^{x_i \ln\frac{|z_i|}{|z|}}e^{-\frac{\pi}{2}|y_i|}e^{C_2 |x|} < \\ < C_4\prod_i e^{x_i \ln|z|}e^{-\frac{\pi}{2}|y_i|}e^{C_3 |x|} <
     C_4 e^{(\sum_i x_i)\ln |z|} e^{-C_5|y|}e^{C_3|x|},
 \end{multline}
 where the second inequality is due to the fact that $\ln|z_i|/|z|$ is bounded from above, and $x_i\ln|z_i|/|z|$ is bounded for small $z_i$. Moreover, $\sum_i x_i < 0$ for $\alpha$ in an open set including $\alpha=0$, because $\delta$ is in the effective cone. We then use the estimate $\ln |x| \le \ln |z|$ to write
 \[
    e^{(\sum_i x_i)\ln|z|} \le e^{(\sum_i x_i) \ln |x|}.
 \]
 Finally, we can estimate the integral~\eqref{eq:intEstimate1}:
 \begin{equation}
     I < C_6\int_{\delta + \sqrt{-1}(\fg_I)_{\bR}} e^{\sum_a t_a \delta_a}e^{(\sum_i x_i) \ln |x|}e^{-C_5|y|}e^{C_3|x|} \; d^{k-|I|} y,
 \end{equation}
 where we used that the integral over a compact torus is bounded by the maximum of the integrand times the volume of the torus,
 and $d^{k-|I|}y$ is a shift invariant volume form. We can compute the integral using polar coordinates with the result:
 \[
    I < C_7 e^{(\sum_i x_i) \ln |x|}e^{\sum_a t_a\delta_a + C_3|x|}.
 \]
 The estimate of the theorem follows after relabeling of the constants.
\end{proof}

\bibliographystyle{alphaurl}
\bibliography{Bibliography}
\end{document}